\documentclass{article}
\usepackage[utf8]{inputenc}
\usepackage{amsmath,amssymb}
\usepackage{amsthm}
\usepackage{parskip}
\usepackage{hyperref}
\usepackage[margin=1in]{geometry}
\usepackage{authblk} 
\usepackage{mathtools} 
\mathtoolsset{showonlyrefs,showmanualtags}

\usepackage{bbm}
\usepackage{mathrsfs}
\usepackage{graphicx}
\usepackage{multicol}
\usepackage{tikz}\usetikzlibrary{positioning,chains,fit,shapes,calc}
\newcommand*\circled[1]{\tikz[baseline=(char.base)]{
            \node[shape=circle,draw,inner sep=2pt] (char) {#1};}}

\mathtoolsset{showonlyrefs,showmanualtags}
\DeclareRobustCommand{\stirling}{\genfrac\{\}{0pt}{}}

\makeatletter
\def\thm@space@setup{%
  \thm@preskip=\parskip \thm@postskip=0pt
}
\makeatother

\hypersetup{
    colorlinks=true,
    linkcolor=blue,
    citecolor=darkgray
}

\usepackage{mathrsfs}
\usepackage{tikz}
\newtheorem{thm}{Theorem}[section]
\newtheorem{Q}{Question}[section]
\newtheorem{lemma}{Lemma}[section]
\newtheorem{cor}{Corollary}[section]
\newtheorem{remark}{Remark}
\theoremstyle{definition}

\usepackage[colorinlistoftodos]{todonotes}

\newcommand{\e}{\varepsilon}
\newcommand{\cV}{\mathcal{V}}
\newcommand{\cR}{\mathcal{R}}
\newcommand{\cD}{\mathcal{D}}

\newcommand{\C}{\mathbb{C}}

\newcommand{\EE}{\mathbb{E}}

\newcommand{\be}{\begin{equation}}
\newcommand{\ee}{\end{equation}}

\DeclareRobustCommand{\stirling}{\genfrac\{\}{0pt}{}}

\title{The distribution of the length of the longest path in random acyclic orientations of a complete bipartite graph}

\author[1]{Jessica Khera}
\author[2]{Erik Lundberg\thanks{The author acknowledges support from the Simons Foundation (grant 712397).}}
\affil[1]{Florida International University, jkhera@fiu.edu}
\affil[2]{Florida Atlantic University, elundber@fau.edu}

\begin{document}

\maketitle

\begin{abstract} 
Randomly sampling an acyclic orientation on the complete bipartite graph $K_{n,k}$ with parts of size $n$ and $k$, we investigate the length of the longest path.  We provide a probability generating function for the distribution of the longest path length, and we use Analytic Combinatorics to perform asymptotic analysis of the probability distribution in the case of equal part sizes $n=k$ tending toward infinity.  We show that the distribution is asymptotically Gaussian, and we obtain precise asymptotics for the mean and variance.  These results address a question asked by Peter J. Cameron.
\noindent 
    \\ Keywords: bipartite graph, directed graph, random graph, acyclic orientation, poly-Bernoulli numbers, lonesum matrices, generating function, analytic combinatorics, asymptotics. 
\end{abstract}



\section{Introduction}

An \textit{acyclic orientation} of a graph is an assignment of a direction to each edge in the graph in a way that does not form any directed cycles.
The problem of studying the distribution of the length of the longest (directed) path for an acyclic orientation sampled uniformly at random from the collection of all acyclic orientations on a fixed graph has been stated by Peter J. Cameron \cite{CameronSlides}, with particular attention given to the setting when the underlying graph is a complete bipartite graph.

\begin{Q}[P. J. Cameron]\label{q:bipartite}
Given positive integers $n$ and $k$, what is the distribution of the length of the longest path in a random acyclic orientation of $K_{n,k}$ (as usual $K_{n,k}$ denotes the complete bipartite graph with parts of sizes $n$ and $k$)?
\end{Q}

In the current paper, we focus on addressing this probabilistic question with exact (enumerative) and asymptotic results.

One source of motivation for studying longest paths in directed acyclic graphs comes from applied settings where directed acyclic graphs arise, for instance, as dependency graphs for systems of tasks.  In fact, the industry-relevant problem of finding a critical path for schedule planning reduces to finding a longest path in an acyclic directed graph \cite{digraphBook}.  From an algorithmic perspective,  the problem of finding the longest directed path in a given directed acyclic graph can be solved in linear time \cite{digraphBook} (whereas the longest path problem for general graphs is NP-hard).  

We also note that the study of acyclic orientations of graphs is a well-developed and active research area. A few specific examples are given here. The Gallai-Hasse-Roy-Vitaver theorem \cite{Nesetril} finds that the chromatic number of a graph is one more than the length of the longest path of an acyclic orientation chosen to minimize this length. Stanley showed in 1973 \cite{StanleyAO} that the number of acyclic orientations of a graph $G$ with order $n$ is $(-1)^n \chi_G(-1)$, where $\chi_G$ denotes the chromatic polynomial of $G$, and this result was later used to prove that determining the number of acyclic orientations of a graph is an \#P-complete problem \cite{Linial}. In the 1980s, Johnson \cite{Johnson} studied how acyclic orientations of a network with a unique source can be used to compute its reliability. Other research related to acyclic orientations can also be found in \cite{ayyer}, \cite{barbosa}, \cite{gebhard}, \cite{jensen}, \cite{reidys}, and \cite{west}.

In the current paper, we address Question \ref{q:bipartite} with enumerative and asymptotic results.  First, we provide a description of the distribution of the longest path in the form of a probability generating function obtained as a corollary (see Corollary \ref{cor:PGF} below) of the following enumerative result.  As our second main result, we provide an asymptotic limit law for the distribution of the longest path when the part sizes satisfy $n=k \rightarrow \infty$ (see Theorem \ref{thm:asymp} below), showing that the distribution is asymptotically Gaussian as well as providing precise asymptotics for the mean and variance.

\begin{thm}\label{thm:main}
Let $G_{n,k}(\ell)$ denote the number of acyclic orientations on the complete graph $K_{n,k}$ with length of longest path equal to $\ell$. The associated generating function  $\displaystyle F(x,y,u) = \sum_{n,k,\ell \geq 0} G_{n,k}(\ell)u^{\ell}\frac{x^n}{n!}\frac{y^k}{k!}$ satisfies $$F(x,y,u)= \frac{e^{x+y} - (u-1)^2(e^{x}-1)(e^{y}-1)}{1-u^2(e^x-1)(e^y-1)}.$$
\end{thm}

We note that the evaluation $F (x, y, 1)$ at $u = 1$ recovers the known generating function $$B(x,y) = \frac{e^{x+y}}{1-(e^x-1)(e^y-1)}$$ for the total number of acyclic orientations which are enumerated by the poly-Bernoulli numbers, see Section \ref{sec:polyBernoulli} below for further discussion.

Returning to Question \ref{q:bipartite}, we can use the enumerative generating function provided by Theorem \ref{thm:main} to derive, by a simple standard procedure the desired \emph{probability} generating function (PGF) $p_{n,k}(u)$ whose $u^\ell$-coefficient is (by definition) the probability that an acyclic orientation, sampled uniformly at random, has length of its longest path being equal to $\ell$.
Namely, we have the following corollary.


\begin{cor}\label{cor:PGF}
The probability generating polynomial $p_{n,k}(u)$ of the length of the longest path is given by
$$p_{n,k}(u)= \frac{[x^ny^k]F(x,y,u)}{[x^ny^k]F(x,y,1)},$$
with $F(x,y,u)$ as in Theorem \ref{thm:main}.
\end{cor}

\begin{remark}
Using Corollary \ref{cor:PGF} along with simple properties of probability generating functions, we obtain the mean and variance of the length of the longest path as follows \cite[Ch. III and Appendix A.3]{Flajolet}.
The expectation of the length of the longest path is given by
 $$\frac{\partial p_{n,k}(u)}{\partial u}  \rvert_{u=1} = \frac{[x^ny^k]  F_u(x,y,1) }{[x^ny^k] F(x,y,1)}, $$
and the variance of the length of the longest path is given by
 $$\frac{\partial^2 p_{n,k}(u)}{\partial u ^2}  \rvert_{u=1} +  \frac{\partial p_{n,k}(u)}{\partial u}  \rvert_{u=1} - \left( \frac{\partial p_{n,k}(u)}{\partial u} \rvert_{u=1} \right)^2 . $$
\end{remark}

Notice that in an acyclic orientation of the complete bipartite graph $K_{n,k}$, the length $\ell$ of the longest path cannot exceed $2\min\{n,k\}-1$, so the probability generating function $p_{n,k}(u)$ is a polynomial of this degree. 

\noindent{\bf Example.} Let us consider the simple case $n=k=2$ where it is easy to sketch out all the possible acyclic orientations, which are displayed in Figure \ref{example}. In this case, we have the probability generating function $$p_{2,2} (u) = \frac{1}{7}u +\frac{2}{7}u^2 +\frac{4}{7}u^3,$$
i.e., there is probability $\frac{1}{7}$ that the longest path has length $\ell=1$, probability $\frac{2}{7}$ that the longest path has length $\ell=2$, and probability $\frac{4}{7}$ that the longest path has length $\ell=3$.  The mean is then 
$$ \partial_u p_{2,2}(u) \vert_{u=1} =  \frac{1}{7} +  \frac{2}{7} \cdot 2 +  \frac{4}{7}\cdot 3 = \frac{17}{7}.$$

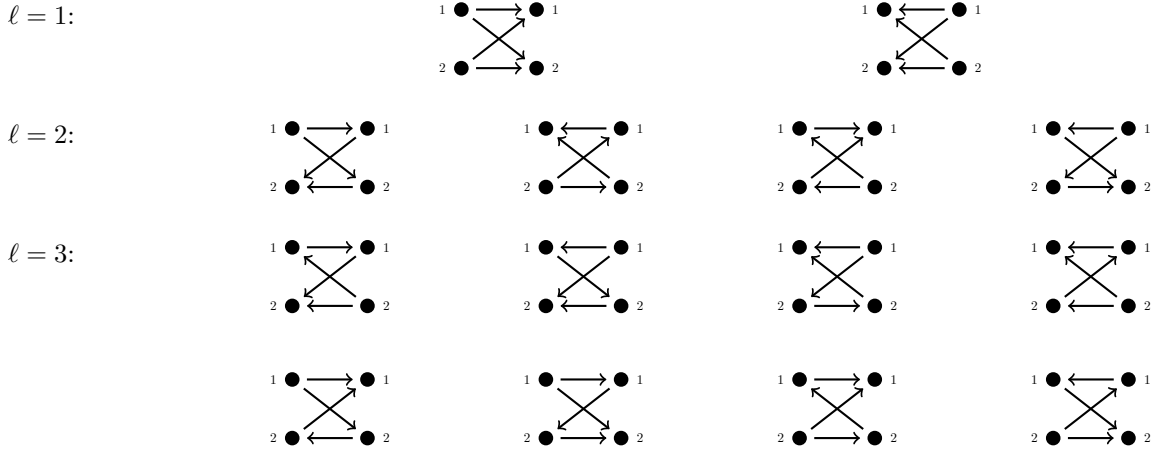
\begin{figure}[!ht]
\begin{multicols}{3}

$\ell=1$: 
\\
\begin{tikzpicture}[thick,
  every node/.style={draw,circle,scale=0.5},
  fsnode/.style={fill=black},
  ssnode/.style={fill=black},
  every fit/.style={ellipse,draw,inner sep=2mm,text width=2mm},
  ->,shorten >= 3pt,shorten <= 3pt, scale=0.5
]

\begin{scope}[start chain=going below,node distance=6mm]
\foreach \i in {1,2}
  \node[fsnode,on chain] (f\i) [label=left: \i] {};
\end{scope}

\begin{scope}[xshift=2cm,start chain=going below,node distance=6mm]
\foreach \i in {1,2}
  \node[ssnode,on chain] (s\i) [label=right: \i] {};
\end{scope}


\draw (f1) -- (s1);
\draw (f1) -- (s2);
\draw (f2) -- (s1);
\draw (f2) -- (s2);

\end{tikzpicture}
 
\begin{tikzpicture}[thick,
  every node/.style={draw,circle,scale=0.5},
  fsnode/.style={fill=black},
  ssnode/.style={fill=black},
  every fit/.style={ellipse,draw,inner sep=2mm,text width=2mm},
  ->,shorten >= 3pt,shorten <= 3pt, scale=0.5
]

\begin{scope}[start chain=going below,node distance=6mm]
\foreach \i in {1,2}
  \node[fsnode,on chain] (f\i) [label=left: \i] {};
\end{scope}

\begin{scope}[xshift=2cm,start chain=going below,node distance=6mm]
\foreach \i in {1,2}
  \node[ssnode,on chain] (s\i) [label=right: \i] {};
\end{scope}


\draw (s1) -- (f1);
\draw (s1) -- (f2);
\draw (s2) -- (f1);
\draw (s2) -- (f2);

\end{tikzpicture} 
 \end{multicols}  

\begin{multicols}{5}
$\ell=2$: 

\begin{tikzpicture}[thick,
  every node/.style={draw,circle,scale=0.5},
  fsnode/.style={fill=black},
  ssnode/.style={fill=black},
  every fit/.style={ellipse,draw,inner sep=2mm,text width=2mm},
  ->,shorten >= 3pt,shorten <= 3pt, scale=0.5
]

\begin{scope}[start chain=going below,node distance=6mm]
\foreach \i in {1,2}
  \node[fsnode,on chain] (f\i) [label=left: \i] {};
\end{scope}

\begin{scope}[xshift=2cm,start chain=going below,node distance=6mm]
\foreach \i in {1,2}
  \node[ssnode,on chain] (s\i) [label=right: \i] {};
\end{scope}


\draw (f1) -- (s1);
\draw (f1) -- (s2);
\draw (s1) -- (f2);
\draw (s2) -- (f2);

\end{tikzpicture} 

\begin{tikzpicture}[thick,
  every node/.style={draw,circle,scale=0.5},
  fsnode/.style={fill=black},
  ssnode/.style={fill=black},
  every fit/.style={ellipse,draw,inner sep=2mm,text width=2mm},
  ->,shorten >= 3pt,shorten <= 3pt, scale=0.5
]

\begin{scope}[start chain=going below,node distance=6mm]
\foreach \i in {1,2}
  \node[fsnode,on chain] (f\i) [label=left: \i] {};
\end{scope}

\begin{scope}[xshift=2cm,start chain=going below,node distance=6mm]
\foreach \i in {1,2}
  \node[ssnode,on chain] (s\i) [label=right: \i] {};
\end{scope}


\draw (s1) -- (f1);
\draw (f2) -- (s1);
\draw (s2) -- (f1);
\draw (f2) -- (s2);

\end{tikzpicture} 

\begin{tikzpicture}[thick,
  every node/.style={draw,circle,scale=0.5},
  fsnode/.style={fill=black},
  ssnode/.style={fill=black},
  every fit/.style={ellipse,draw,inner sep=2mm,text width=2mm},
  ->,shorten >= 3pt,shorten <= 3pt, scale=0.5
]

\begin{scope}[start chain=going below,node distance=6mm]
\foreach \i in {1,2}
  \node[fsnode,on chain] (f\i) [label=left: \i] {};
\end{scope}

\begin{scope}[xshift=2cm,start chain=going below,node distance=6mm]
\foreach \i in {1,2}
  \node[ssnode,on chain] (s\i) [label=right: \i] {};
\end{scope}


\draw (f1) -- (s1);
\draw (f2) -- (s1);
\draw (s2) -- (f1);
\draw (s2) -- (f2);

\end{tikzpicture} 

\begin{tikzpicture}[thick,
  every node/.style={draw,circle,scale=0.5},
  fsnode/.style={fill=black},
  ssnode/.style={fill=black},
  every fit/.style={ellipse,draw,inner sep=2mm,text width=2mm},
  ->,shorten >= 3pt,shorten <= 3pt, scale=0.5
]

\begin{scope}[start chain=going below,node distance=6mm]
\foreach \i in {1,2}
  \node[fsnode,on chain] (f\i) [label=left: \i] {};
\end{scope}

\begin{scope}[xshift=2cm,start chain=going below,node distance=6mm]
\foreach \i in {1,2}
  \node[ssnode,on chain] (s\i) [label=right: \i] {};
\end{scope}


\draw (s1) -- (f1);
\draw (s1) -- (f2);
\draw (f1) -- (s2);
\draw (f2) -- (s2);

\end{tikzpicture} 

\end{multicols}

\begin{multicols}{5}
$\ell=3$: 

\vspace{1in}

\begin{tikzpicture}[thick,
  every node/.style={draw,circle,scale=0.5},
  fsnode/.style={fill=black},
  ssnode/.style={fill=black},
  every fit/.style={ellipse,draw,inner sep=2mm,text width=2mm},
  ->,shorten >= 3pt,shorten <= 3pt, scale=0.5
]

\begin{scope}[start chain=going below,node distance=6mm]
\foreach \i in {1,2}
  \node[fsnode,on chain] (f\i) [label=left: \i] {};
\end{scope}

\begin{scope}[xshift=2cm,start chain=going below,node distance=6mm]
\foreach \i in {1,2}
  \node[ssnode,on chain] (s\i) [label=right: \i] {};
\end{scope}


\draw (f1) -- (s1);
\draw (s2) -- (f2);
\draw (s2) -- (f1);
\draw (s1) -- (f2);

\end{tikzpicture} 
\vspace{0.25in}

\begin{tikzpicture}[thick,
  every node/.style={draw,circle,scale=0.5},
  fsnode/.style={fill=black},
  ssnode/.style={fill=black},
  every fit/.style={ellipse,draw,inner sep=2mm,text width=2mm},
  ->,shorten >= 3pt,shorten <= 3pt, scale=0.5
]

\begin{scope}[start chain=going below,node distance=6mm]
\foreach \i in {1,2}
  \node[fsnode,on chain] (f\i) [label=left: \i] {};
\end{scope}

\begin{scope}[xshift=2cm,start chain=going below,node distance=6mm]
\foreach \i in {1,2}
  \node[ssnode,on chain] (s\i) [label=right: \i] {};
\end{scope}


\draw (f1) -- (s1);
\draw (f1) -- (s2);
\draw (f2) -- (s1);
\draw (s2) -- (f2);

\end{tikzpicture} 

\begin{tikzpicture}[thick,
  every node/.style={draw,circle,scale=0.5},
  fsnode/.style={fill=black},
  ssnode/.style={fill=black},
  every fit/.style={ellipse,draw,inner sep=2mm,text width=2mm},
  ->,shorten >= 3pt,shorten <= 3pt, scale=0.5
]

\begin{scope}[start chain=going below,node distance=6mm]
\foreach \i in {1,2}
  \node[fsnode,on chain] (f\i) [label=left: \i] {};
\end{scope}

\begin{scope}[xshift=2cm,start chain=going below,node distance=6mm]
\foreach \i in {1,2}
  \node[ssnode,on chain] (s\i) [label=right: \i] {};
\end{scope}


\draw (s1) -- (f1);
\draw (f1) -- (s2);
\draw (s1) -- (f2);
\draw (s2) -- (f2);

\end{tikzpicture} 
\vspace{0.25in}

\begin{tikzpicture}[thick,
  every node/.style={draw,circle,scale=0.5},
  fsnode/.style={fill=black},
  ssnode/.style={fill=black},
  every fit/.style={ellipse,draw,inner sep=2mm,text width=2mm},
  ->,shorten >= 3pt,shorten <= 3pt, scale=0.5
]

\begin{scope}[start chain=going below,node distance=6mm]
\foreach \i in {1,2}
  \node[fsnode,on chain] (f\i) [label=left: \i] {};
\end{scope}

\begin{scope}[xshift=2cm,start chain=going below,node distance=6mm]
\foreach \i in {1,2}
  \node[ssnode,on chain] (s\i) [label=right: \i] {};
\end{scope}


\draw (f1) -- (s1);
\draw (f1) -- (s2);
\draw (s1) -- (f2);
\draw (f2) -- (s2);

\end{tikzpicture} 

\begin{tikzpicture}[thick,
  every node/.style={draw,circle,scale=0.5},
  fsnode/.style={fill=black},
  ssnode/.style={fill=black},
  every fit/.style={ellipse,draw,inner sep=2mm,text width=2mm},
  ->,shorten >= 3pt,shorten <= 3pt, scale=0.5
]

\begin{scope}[start chain=going below,node distance=6mm]
\foreach \i in {1,2}
  \node[fsnode,on chain] (f\i) [label=left: \i] {};
\end{scope}

\begin{scope}[xshift=2cm,start chain=going below,node distance=6mm]
\foreach \i in {1,2}
  \node[ssnode,on chain] (s\i) [label=right: \i] {};
\end{scope}


\draw (s1) -- (f1);
\draw (s1) -- (f2);
\draw (s2) -- (f1);
\draw (f2) -- (s2);

\end{tikzpicture} 
\vspace{0.25in}

\begin{tikzpicture}[thick,
  every node/.style={draw,circle,scale=0.5},
  fsnode/.style={fill=black},
  ssnode/.style={fill=black},
  every fit/.style={ellipse,draw,inner sep=2mm,text width=2mm},
  ->,shorten >= 3pt,shorten <= 3pt, scale=0.5
]

\begin{scope}[start chain=going below,node distance=6mm]
\foreach \i in {1,2}
  \node[fsnode,on chain] (f\i) [label=left: \i] {};
\end{scope}

\begin{scope}[xshift=2cm,start chain=going below,node distance=6mm]
\foreach \i in {1,2}
  \node[ssnode,on chain] (s\i) [label=right: \i] {};
\end{scope}


\draw (f1) -- (s1);
\draw (s2) -- (f1);
\draw (f2) -- (s2);
\draw (f2) -- (s1);

\end{tikzpicture} 

\begin{tikzpicture}[thick,
  every node/.style={draw,circle,scale=0.5},
  fsnode/.style={fill=black},
  ssnode/.style={fill=black},
  every fit/.style={ellipse,draw,inner sep=2mm,text width=2mm},
  ->,shorten >= 3pt,shorten <= 3pt, scale=0.5
]

\begin{scope}[start chain=going below,node distance=6mm]
\foreach \i in {1,2}
  \node[fsnode,on chain] (f\i) [label=left: \i] {};
\end{scope}

\begin{scope}[xshift=2cm,start chain=going below,node distance=6mm]
\foreach \i in {1,2}
  \node[ssnode,on chain] (s\i) [label=right: \i] {};
\end{scope}


\draw (s1) -- (f1);
\draw (f2) -- (s1);
\draw (s2) -- (f1);
\draw (s2) -- (f2);

\end{tikzpicture} 
\vspace{0.25in}

\begin{tikzpicture}[thick,
  every node/.style={draw,circle,scale=0.5},
  fsnode/.style={fill=black},
  ssnode/.style={fill=black},
  every fit/.style={ellipse,draw,inner sep=2mm,text width=2mm},
  ->,shorten >= 3pt,shorten <= 3pt, scale=0.5
]

\begin{scope}[start chain=going below,node distance=6mm]
\foreach \i in {1,2}
  \node[fsnode,on chain] (f\i) [label=left: \i] {};
\end{scope}

\begin{scope}[xshift=2cm,start chain=going below,node distance=6mm]
\foreach \i in {1,2}
  \node[ssnode,on chain] (s\i) [label=right: \i] {};
\end{scope}


\draw (s1) -- (f1);
\draw (f2) -- (s1);
\draw (f1) -- (s2);
\draw (f2) -- (s2);

\end{tikzpicture} 
\end{multicols}
\caption{There are $B_{2,2} = 14$ possible acyclic orientations of $K_{2,2}$. 
Here they are arranged according to the length $\ell$ of the longest path.  Two of them have $\ell =1$, four of them have $\ell=2$, and the remaining eight orientations have $\ell=3$.}
\label{example}
\end{figure}

Plotting the probability distribution for several choices of $n=k$, i.e., plotting the coefficient $[u^\ell] p_{n,k}(u)$ with respect to $\ell$, we notice a characteristic bell-shape (apparent unimodality and apparent log-concavity) that may lead one to suspect that the distribution is asymptotically Gaussian (see Figure \ref{fig:plots}). This turns out to be true, as we show in our main result, Theorem \ref{thm:asymp} below.  We note that showing asymptotic Gaussianity in this case is not as simple as applying the classical central limit theorem (as there is no obvious way to view our statistic as a sum of independent random variables); rather, we establish a central limit type theorem (while also providing precise asymptotics for the mean and variance) using methods of Analytic Combinatorics to prove the following result.

\begin{figure}
\centering
\includegraphics[scale=0.4]{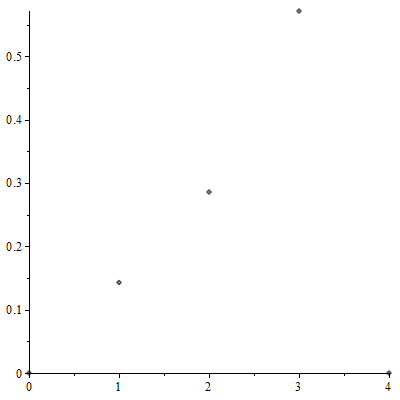}
\includegraphics[scale=0.4]{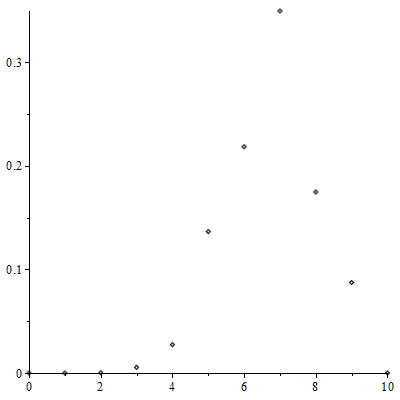}
\includegraphics[scale=0.36]{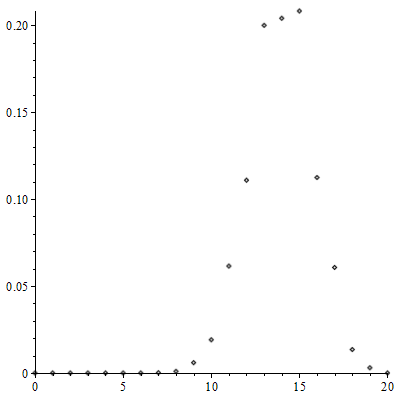} 
\includegraphics[scale=0.36]{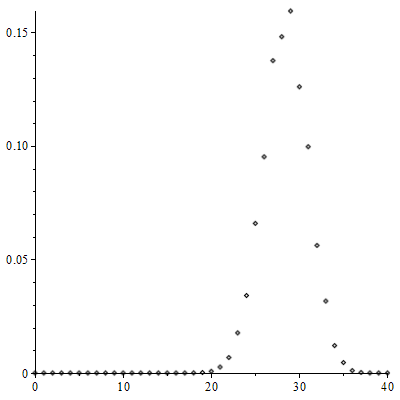}
\includegraphics[scale=0.36]{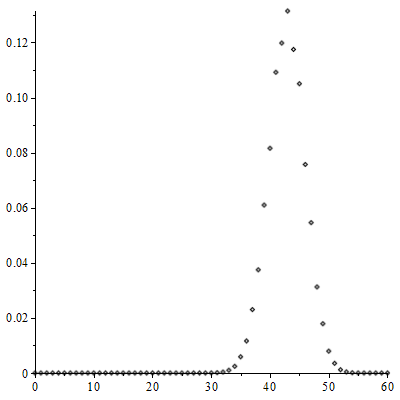}
\caption{Plots showing the probabilities of the longest path having length $\ell$, for bipartite graphs with parts of equal size $n=k$ for $n=2$, $n=5$, $n=10$, $n=20$, and $n=30$.}
\label{fig:plots}
\end{figure}

\begin{thm}\label{thm:asymp}
Let $X_n$ denote the random variable equal to the length of the longest path in a random acyclic orientation of $K_{n,n}$. Let $\mu_n \coloneqq \mathbb{E} X_n$ and $\sigma_n^2 \coloneqq \sqrt{\mathbb{E}(X_n-\mathbb{E}X_n)^2}$. 
Then the random variable $$Y_n \coloneqq \frac{X_n-\mu_n}{\sigma_n^2}$$ converges in distribution to the standard Gaussian as $n \rightarrow \infty$. 
Moreover, the mean $\mu_n$ and variance $\sigma_n^2$ of $X_n$ satisfy the following asymptotics $$\mu_n =  \frac{n}{\log 2} + \frac{8(\log 2)^2-9 \log 2 + 2}{4 \log 2 (1-\log 2)} + O(n^{-1}), \quad \text{as \, } n \rightarrow \infty$$ and 
$$\sigma_n^2 = \frac{n (1-\log 2)}{2 (\log 2)^2} + \frac{-2 (\log 2)^4 + (\log 2)^3 + 2 (\log 2)^2 - 6 \log 2 + 2}{8 (\log 2)^2 (1- \log 2)^2} + O(n^{-1}), \quad \text{as \, } n \rightarrow \infty.$$ 
\end{thm}

Notice that, since the variance is asymptotically proportional to the mean, an application of Chebyshev's inequality shows that $X_n$ concentrates near its mean. Indeed, for an arbitrary $\e>0$ we have
\be
P(|X_n - \mu_n|> n^{\frac{1}{2}+\e}) \leq \frac{\sigma_n^2}{n^{1+2\e}} = O(n^{-2\e}) = o(1), \quad (n \rightarrow \infty).
\ee
In particular, taking $0<\e<1/2$ we have that the length of the longest path is asymptotic to its mean, i.e., $X_n = \mu_n (1+o(1))$, with high probability.

\subsection{Other enumeration problems related to the poly-Bernoulli numbers}\label{sec:polyBernoulli}

Theorem \ref{thm:main} can be viewed as a refinement of the following enumerative result due to Cameron, Glass, and Schumacher \cite{CameronPreprint}.  The statement of the result uses the standard notation $\stirling{n}{k}$ (read ``$n$ brace $k$") to denote the \textit{Stirling numbers of the second kind} that count the number of ways to partition an $n$ element set into $k$ nonempty subsets \cite{stanley}.  
\begin{thm}
The number of acyclic orientations on the complete bipartite graph $K_{n,k}$ is given by the poly-Bernoulli number
\be\label{eq:formula}
B_{n,k} = \sum_{m \geq 0} (m!)^2 \stirling{n+1}{m+1} \stirling{k+1}{m+1}.
\ee
\end{thm}
Table \ref{table:Bnk} shows some values of $B_{n,k}$.

Within the number-theoretic setting in which the poly-Bernoulli numbers were introduced, $B_{n,k}$ is actually referred to as a poly-Bernoulli number with one index negative $B_{n,k} = B_n^{(-k)}$.  Namely, the traditional poly-Bernoulli numbers $B_n^{(k)}$ were introduced by Kaneko \cite{kaneko} as a generalization of the Bernoulli numbers (with $B_n^{(1)}$ being the usual Bernoulli numbers), defined  as $$\frac{Li_k(1-e^{-x})}{1-e^{-x}}=\sum_{n=0}^{\infty} B_n^{(k)} \frac{x^n}{n!}$$ where $Li$ is the polylogarithm. They were later explored in connection to zeta functions and multiple zeta values~\cite{ArakawaKaneko1999} by Arakawa and Kaneko who also established \cite{Kaneko2} the above formula \eqref{eq:formula} for $B_n^{(-k)}$. For further more recent number-theoretic developments related to poly-Bernoulli numbers see \cite{benyi3}, \cite{benyi4}, \cite{benyi5}, \cite{BrewbakerPaper}, \cite{kaneko}, and \cite{matsusaka}.  From a combinatorial perspective, we are solely interested in the poly-Bernoulli numbers with negative $k$ indices, $B_{n,k} \coloneqq B_n^{(-k)}$, which we will simply refer to as poly-Bernoulli numbers.  

It is noteworthy that the counting array $B_{n,k}$ has appeared as the solution to many enumerative problems in addition to the above-mentioned case of acyclic orientations on complete bipartite graphs.
Several combinatorial interpretations of $B_{n,k}$ are presented by B\'enyi and Hajnal in \cite{benyi1} and \cite{benyi2}.

One such interpretation is that of \emph{Callan permutations}, which are permutations of $\{1,\dots,n+k\}$ in which the elements of the sets $\{1,\dots,n\}$ and $\{n+1,\dots,n+k\}$ each appear in increasing order. Another is that of \emph{Vesztergombi permutations}, which are permutations $\pi$ of $\{1,\dots,n+k\}$ satisfying $-k \leq \pi(i)-i \leq n$. The enumeration of such permutations originally appeared in Vesztergombi~\cite{VesztergombiPerm}, using generating functions.

For the purposes of the current paper, the most important combinatorial class, besides acyclic orientations, enumerated by the poly-Bernoulli numbers is the class of lonesum matrices. As shown by Brewbaker \cite{BrewbakerThesis}, the number of $n \times k$ lonesum matrices is given by $B_{n,k}$ (see Theorem \ref{thm:polyB}).  Another interpretation which is somewhat related to the class of lonesum matrices is that of so-called \emph{$\Gamma$-free matrices}. These are $n \times k$ $(0,1)$-matrices that avoid the $2 \times 2$ submatrices in which 1s appear in a configuration forming the shape of the letter $\Gamma$, i.e., those that avoid $\begin{bmatrix}1 & 1 \\ 1 & 0\end{bmatrix}$ and $\begin{bmatrix}1 & 1 \\ 1 & 1\end{bmatrix}$ as submatrices. 

The poly-Bernoulli numbers $B_{n,n}$ also appear in Algebraic Statistics via the theory of matrix Schubert varieties~\cite{FRS}, where they count the number of strata in a certain stratification of the space of $n \times n$ matrices.

Additional combinatorial interpretations, including so-called ``parades'', are pointed out in the recent survey \cite{KnuthPreprint} by Knuth, who, based on their variety of combinatorial properties,  deems poly-Bernoulli numbers to be a ``a counterexample to the hypothesis that all of the important `special numbers'
were discovered long, long ago.''

\begin{table}\label{table:Bnk}
\centering
\begin{tabular}{| c || c | c | c | c | c | c | c |}
\hline 
n/k & 0 & 1 & 2 & 3 & 4 & 5 & 6\\ 
\hline\hline 
0 & 1 & 1 & 1 & 1 & 1  & 1 & 1\\
\hline
1 & 1 & 2 & 4 & 8 & 16  & 32 & 64\\ 
\hline 
2 & 1 & 4 & 14 & 46 & 146 & 454 & 1394\\ 
\hline 
3 & 1 & 8 & 46 & 230 & 1066 & 4718 & 20266\\
\hline 
4 & 1 & 16 & 146 & 1066 & 6906 & 41506 & 237686\\
\hline
5 & 1 & 32 & 454 & 4718 & 41506 & 329462 & 2441314\\
\hline
6 & 1 & 64 & 1394 & 20266 & 237686 & 2441314 & 22934774 \\
\hline
\end{tabular}
\caption{Some values of the poly-Bernoulli numbers, $B_{n,k}$.}
\end{table}

\subsection{Sketch of the main ideas underlying the proofs}

Our proof of the enumerative result Theorem \ref{thm:main} is based on elementary combinatorics along with basic manipulation of generating functions, and by reviewing a few preliminary results in Section \ref{sec:prelim} we are able to keep the presentation of the proof self-contained, and a reader interested in understanding our enumerative (non-asymptotic) results may dive right into the proofs.  On the other hand, the proof of Theorem \ref{thm:asymp} includes an application of two rather high-level results from Analytic Combinatorics, with each of those results being treated essentially as a black box.  The main work needed for applying those results consists of establishing some technical lemmas whose proofs, while interesting for the geometric and complex analytic techniques they introduce to this setting, do not shed light on the essential ideas behind the convergence to a Gaussian limit.  Thus, in order to provide some basic insight, let us present in broad strokes the proof here while omitting the technicalities in favor of conveying just the main ideas underlying the results from Analytic Combinatorics.

Let $p_n(u)$ denote (for $k=n$) the probability generating function provided by Corollary \ref{cor:PGF}.  Proving the desired Gaussian limit hinges on establishing (uniformly near $u=1$) a so-called ``quasi-power'' asymptotic condition, namely,
\be\label{eq:quasipowcond}
p_n(u)=A(u) \cdot B(u)^{n} \left( 1+ O \left( n^{-1} \right) \right),
\ee
with $A(u), B(u)$ analytic near $u=1$.
That this quasi-power condition is sufficient for the desired Gaussian limit follows from a result in Analytic Combinatorics (the first of the two ``high-level'' results mentioned above) which can be proved using a version of the classical L\'evy continuity principle, which says that pointwise convergence of the characteristic function (Fourier Transform) of a sequence of random variables implies convergence in distribution.  To illustrate, let us sketch a similar application of L\'evy's continuity principle within a simpler setting, the case of the binomial random variable $X_n$ for the number of heads in $n$ trials of a fair coin.  The PGF is $p_n(u) = (\frac{1+u}{2})^n$ which is already in the form of a quasi-power (in fact, an exact power).  The mean is $\mu_n=n/2$, and the variance is $\sigma_n^2=n/4$.  
The characteristic function of $X_n$ is $\EE e^{it X_n} = p_n(e^{i t}) = (\frac{1+e^{i t}}{2})^n$, and the characteristic function of the \emph{rescaled} random variable $(X_n - \mu_n)/\sigma_n$ is $\EE e^{it (X_n-\mu_n)/\sigma_n} = e^{-it\sqrt{n}} \EE e^{it X_n/\sigma_n} =  (\frac{e^{-it/\sqrt{n}}+e^{i t/\sqrt{n}}}{2})^n = \cos^n(t/\sqrt{n})$ which converges pointwise as $n \rightarrow \infty$ to $e^{-t^2/2}$, the Fourier inversion of which is the standard Gaussian. 
 Hence, by the continuity principle we recover the well-known fact (which is a special case of the classical central limit theorem) that the binomial distribution is asymptotically Gaussian.
 Similarly, with some additional technicalities in controlling the error terms, the quasi-power condition can be used in conjunction with the continuity principle to establish a Gaussian limit for our statistic of interest, the longest path length.

The question remains as to how we establish the quasi-power condition \eqref{eq:quasipowcond}.  Recall that the PGF $p_n(u)$ is expressed in terms of coefficient extraction of the generating function $F(x,y,u)$.  
\be\label{eq:ratio} p_{n}(u)= \frac{[x^n y^n]F(x,y,u)}{[x^n y^n]F(x,y,1)}\ee
Asymptotic analysis of coefficients of generating functions is a central topic of Analytic Combinatorics, and an available result from the multi-variate setting (the second of the two ``high-level'' results mentioned above, see Theorem \ref{thm:PW} below) applied to the numerator and denominator (treating $u$ as a complex parameter in the numerator case) yields the desired quasi-power condition.  To elaborate, let us give some insight, following expository remarks in \cite{KLM}, on the underlying complex analytic and asymptotic methods by sketching the analysis for the denominator case of \eqref{eq:ratio} (the case of the numerator is similar but with additional technicalities due to the presence of the complex parameter $u$).

We note that the denominator of \eqref{eq:ratio} is nothing other than the poly-Bernoulli number whose asymptotic behavior, as previously obtained by the authors of the current paper together with S. Melczer in \cite{KLM}, is given by

\be\label{eq:KLM}
[x^n y^n] F(x,y,1) = B_{n,n} =  (n!)^2 \sqrt{\frac{1}{n\pi(1-\log 2)}}\left( \frac{1}{\log 2} \right) ^{2n+1} \left( 1 + O(n^{-1})  \right), \quad \text{as \, } n \rightarrow \infty. 
\ee

The starting point for proving \eqref{eq:KLM} is the (iterated) Cauchy integral formula for the coefficient (treating $x$ and $y$ as complex variables and integrating along a circular contour in each of their respective complex planes)
\begin{equation} B_{n,n} = \frac{n!n!}{(2\pi i)^2} \int_{|x|=u} \left(\int_{|y|=v} \frac{1}{e^{-x}+e^{-y}-1} \frac{dy}{y^{n+1}}\right)\frac{dx}{x^{n+1}}, \label{eq:mCIF} \end{equation}
where the radii $u,v$ are initially chosen to be small.  The first step is to deform the circular contours of integration until they pass near the singularity set, and then look for a ``contributing singularity'' about which to localize the main contribution to the integral.
Guided by the symmetry in $F(x,y,1)$ along with the fact that we are considering the diagonal direction $n=k$, we anticipate a contributing singularity to lie on the diagonal set $x=y$, and together with the defining equation for the variety of singularities $e^{-x} + e^{-y}-1 = 0$, this leads us to consider the point $(\log 2, \log 2)$.
Each of the circular contours can be deformed (without encountering singularities, i.e., points $(x,y)\in \C^2$ satisfying $e^{-x}+e^{-y}-1=0$) so that the pair of radii $(u,v)$ in Equation~\eqref{eq:mCIF} can be replaced by $(\log 2, \log 2 - \e)$ for any sufficiently small $\e>0$.  The circular contour of integration $|x|=\log 2$ can be replaced by a small neighborhood $\mathcal{N}$ of $\log 2$ in the circle $|x|=
\log 2$ while introducing an exponentially negligible error (this is related to the fact that the term $1/x^{n+1}$ has modulus comparatively negligibly small outside any neighborhood of the positive real direction). Furthermore, replacing $|y|=\log 2$ by $|y|=\log 2+\epsilon$  (generally referred to as ``pushing the contour past the singularity'') results in an integral which is also exponentially smaller than $B_{n,n}$ (this is related to the uniform smallness of the term $1/y^{n+1}$ along a contour of larger radius). Thus, up to an exponentially negligible error, $B_{n,n}$ is a difference of integrals
\[ \frac{n!n!}{(2\pi i)^2}\int_{x \in \mathcal{N}} \left(\int_{|y|=\log 2 -\epsilon} \frac{1}{e^{-x}+e^{-y}-1} \frac{dy}{y^{n+1}} - \int_{|y|=\log 2 +\epsilon} \frac{1}{e^{-x}+e^{-y}-1} \frac{dy}{y^{n+1}}\right)\frac{dx}{x^{n+1}}.\]
The inner difference of integrals equals the residue of the integrand at the singularity $y=-\log(1-e^{-x})$, and hence, after performing this ($x$-dependent) residue computation, the sequence of interest is asymptotically approximated by the integral
\[ \frac{n!n!}{2\pi i}\int_{x \in \mathcal{N}} \frac{1}{1-e^{-x}} \frac{dx}{x^{n+1}\left(-\log(1-e^{-x})\right)^{n+1}}, \]
whose asymptotics can be obtained using the classical saddle-point method that then yields the desired result \eqref{eq:KLM}.
This outlines the analysis for the asymptotic \eqref{eq:KLM} of the denominator of \eqref{eq:ratio}.
The numerator can be treated similarly (while viewing $u$ as a complex parameter) in order to establish the desired quasi-power condition \eqref{eq:quasipowcond}.
This concludes our sketch of conceptual ideas underlying the Gaussian limit.  The reader interested in further details can read the proofs in Section \ref{sec:AC} while consulting the authoritative texts \cite{Flajolet} by Flajolet and Sedgewick (for background and proof of Theorem \ref{thm:Flaj} on the sufficiency of the quasi-power condition) and \cite{PWM} by Pemantle, Wilson, and Melczer (for background and proof of Theorem \ref{thm:PW} providing the asymptotic used to establish the quasi-power condition).

\subsection{Asymptotic Notation}

Throughout the paper, we use the standard ``big-O'' and ``little-o'' Landau asymptotic notation. For two positive  sequences $a_n$ and $b_n$, we write $b_n=O(a_n)$ if there exists a constant $C$ such that $b_n \le C a_n$, and we write $b_n = o(a_n)$ if $b_n/a_n \rightarrow 0$ as $n \rightarrow \infty$.  
Specifically, we will use these notations to conveniently represent error terms appearing within equations, as is often done, for instance, when the classical Stirling asymptotic $n! \sim \sqrt{2\pi n} \left(\frac{n}{e}\right)^n$ is expressed equivalently as
$$n! = \sqrt{2\pi n} \left(\frac{n}{e}\right)^n\left(1+o(1) \right),$$
or in its version with sharper error estimate
$$n! = \sqrt{2\pi n} \left(\frac{n}{e}\right)^n\left(1+O(n^{-1}) \right),$$
where the expression $o(1)$ (or $O(n^{-1})$ in the case of the second equation) appears in place of a sequence $E_n$ (appropriately viewed as an error term) that satisfies $E_n = o(1)$ (or $E_n=O(n^{-1})$ in the case of the second equation).

\subsection*{Outline of the paper}

We review some preliminary results in Section \ref{sec:prelim} that serve as background and points of reference for our enumerative results. 
We prove our enumerative result Theorem \ref{thm:main} in Section \ref{sec:proofGF}, and we prove our asymptotic result Theorem \ref{thm:asymp} in Section \ref{sec:AC}. We conclude with some open problems in Section \ref{sec:concl}.

\subsection*{Acknowledgements}
The second-named author acknowledges support from the Simons Foundation (grant 712397).  This paper is based in part on results of the first-named author's Ph.D thesis \cite{JessicaThesis}.

\section{Preliminaries}\label{sec:prelim}

In this section, we review some previous results that play an especially important role in the proofs of our results.

\subsection{Acyclic orientations and lonesum matrices}

The proof of Theorem \ref{thm:main} will utilize the following bijection established in \cite{CameronPreprint} between acyclic orientations of the complete bipartite graph $K_{n,k}$ and lonesum matrices of size $n \times k$.  We include the proof here since it will clarify (and serve as a point of reference for) some of the steps in Section \ref{sec:proofGF}.

\begin{thm}\label{thm:bijection} The number of $n \times k$ lonesum matrices is equal to the number of acyclic orientations of $K_{n,k}$. 
\end{thm}

Before proving this result, let us recall the definition of lonesum matrix along with a classification due to Ryser \cite{ryser} in terms of forbidden minors.

A \textit{lonesum matrix} is a (0,1)-matrix that can be uniquely reconstructed from its row sum and column sum vectors. In \cite{ryser}, Ryser proved that a matrix is lonesum if and only if it does not contain
$\begin{bmatrix} 1 & 0 \\ 0 & 1  \end{bmatrix} $ or 
$\begin{bmatrix} 0 & 1 \\ 1 & 0  \end{bmatrix} $ as a submatrix.  One direction of this equivalence is easy to see, namely, if a $(0,1)$ matrix were to contain one of these so-called forbidden minors as a submatrix, it could be replaced by the other forbidden minor without changing any row or column sums, thus violating the uniqueness of the reconstruction. 

Here is an example of a $ 4 \times 4$ lonesum matrix: 
$$\begin{bmatrix}
     1 & 0 & 1 & 0 \\
     1 & 1 & 1 & 1 \\
     0 & 0 & 0 & 0 \\
     1 & 0 & 0 & 0 \\
\end{bmatrix} $$ 
It is easy to verify that this matrix can be uniquely reconstructed given its row sum vector $\begin{bmatrix} 2 \\ 4 \\ 0 \\ 1\\ \end{bmatrix}$ along with its column sum vector $\begin{bmatrix} 3 & 1 & 2 & 1 \end{bmatrix}$.

\begin{proof}[Proof of Theorem \ref{thm:bijection}] Begin by numbering the vertices in the bipartite blocks from $1$ to $n$ in block $A$, and $1$ to $k$ in block $B$. Given an orientation on the graph, we can describe it using a matrix whose $(i,j)$ entry is 1 if the edge from vertex $i$ of $A$ to vertex $j$ of $B$ goes in the direction from $A$ to $B$, and $0$ if the edge is directed from $B$ to $A$ (see Figure \ref{fig:bijection}). The two forbidden submatrices in the classification for lonesum matrices, $\begin{bmatrix} 1 & 0 \\ 0 & 1  \end{bmatrix} $ and 
$\begin{bmatrix} 0 & 1 \\ 1 & 0  \end{bmatrix} $, will then correspond to directed 4-cycles in the graph. Any acyclic orientation of $K_{n,k}$ would result in a matrix which necessarily avoids the forbidden minors and thus must be a lonesum matrix. 

For the converse, we claim that if an orientation of a complete bipartite graph contains no directed 4-cycles, then it contains no directed cycles at all. To see this, suppose that there exists an orientation of a complete bipartite graph in which there are no directed 4-cycles, but there is a longer directed cycle, say $(a_1, b_1, a_2, b_2, \dots , a_k, b_k, a_1)$. Then the edge between $a_1$ and $b_2$ must be oriented towards $b_2$, otherwise we would have a 4-cycle, namely $(a_1, b_1, a_2, b_2, a_1)$. But then we can use this edge from $a_1$ to $b_2$ to obtain a shorter directed cycle $(a_1, b_2, a_3, b_3, \dots , a_k, b_k, a_1)$. Similarly, now we see that the edge between $a_1$ and $b_3$ must also be oriented towards $b_3$ since there are no 4-cycles, and thus we can again obtain a shorter directed cycle $(a_1, b_3, \dots , a_k, b_k, a_1)$. We can continue this process and eventually must arrive at a 4-cycle, which is a contradiction.
\end{proof}

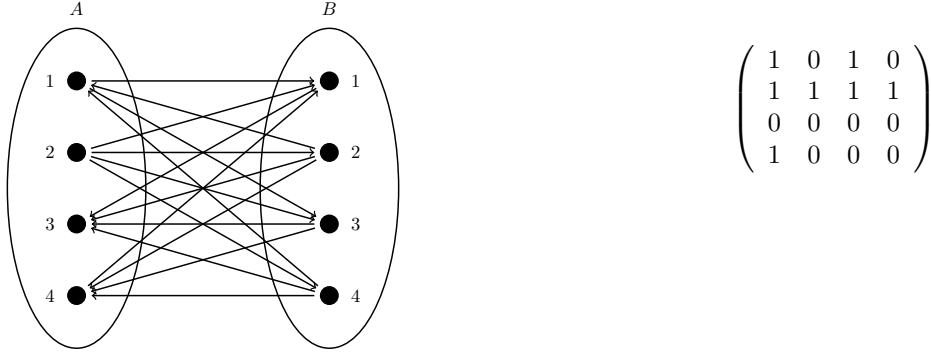
\begin{figure}
\centering
\begin{multicols}{2}
\scalebox{.7}{
\begin{tikzpicture}[scale=1.2,thick,
  every node/.style={draw,circle},
  fsnode/.style={fill=black},
  ssnode/.style={fill=black},
  every fit/.style={ellipse,draw,inner sep=-2pt,text width=2cm},
  ->,shorten >= 3pt,shorten <= 3pt]

\begin{scope}[start chain=going below,node distance=10mm]
\foreach \i in {1,2,...,4}
  \node[fsnode,on chain] (f\i) [label=left: \i] {};
\end{scope}

\begin{scope}[xshift=4cm,start chain=going below,node distance=10mm]
\foreach \i in {1,2,...,4}
  \node[ssnode,on chain] (s\i) [label=right: \i] {};
\end{scope}

\node [black,fit=(f1) (f4),label=above:$A$] {};
\node [black,fit=(s1) (s4),label=above:$B$] {};

\draw (f1) -- (s1);
\draw (s2) -- (f1);
\draw (f1) -- (s3);
\draw (s4) -- (f1);
\draw (f2) -- (s1);
\draw (f2) -- (s2);
\draw (f2) -- (s3);
\draw (f2) -- (s4);
\draw (s1) -- (f3);
\draw (s2) -- (f3);
\draw (s3) -- (f3);
\draw (s4) -- (f3);
\draw (f4) -- (s1);
\draw (s2) -- (f4);
\draw (s3) -- (f4);
\draw (s4) -- (f4);

\end{tikzpicture}}  \\

$$\left( \begin{array}{cccc}
     1 & 0 & 1 & 0 \\
     1 & 1 & 1 & 1 \\
     0 & 0 & 0 & 0 \\
     1 & 0 & 0 & 0 \\
\end{array} \right)  $$

 \end{multicols}

 \caption{An acyclic orientation of the complete bipartite graph $K_{4,4}$, and its corresponding lonesum matrix. The $(i,j)$ entry of the matrix is $1$ if the edge in the graph is oriented from vertex $i$ in $A$ to vertex $j$ in $B$, and $0$ if oriented the opposite way. }\label{fig:bijection}
\end{figure}

We note in passing that the matrix representation of the acyclic orientation described in Theorem \ref{thm:bijection} is essentially the adjacency matrix.  To be precise, it comprises one block in the graph's adjacency matrix written in block form $\begin{bmatrix} 0 & M \\ M^{tb} & 0  \end{bmatrix}$, where each $0$ represents a zero matrix of appropriate size, and $M^{tb}$ is the matrix obtained from $M$ via taking the transpose and a bit-flip.

\subsection{Enumeration of Lonesum Matrices}

In this section, we review the proof of the following counting formula \cite{BrewbakerThesis} for the number of lonesum matrices of size $n \times k$.  Here we use the standard notation 
$\stirling{n}{k}$ (read ``$n$ brace $k$") for the \textit{Stirling numbers of the second kind} that count the number of ways to partition an $n$ element set into $k$ nonempty subsets \cite{stanley}. 

\begin{thm}\label{thm:polyB} The number of lonesum matrices of size $n \times k$ is given by the poly-Bernoulli number $B_{n,k}$. 
\be\label{eq:Bnk} B_{n,k}  = \sum_{m\geq 0} (m!)^2 \stirling{n+1}{m+1} \stirling{k+1}{m+1}.
\ee
\end{thm}

The proof we present below is an elaboration on the proof sketched in \cite{benyi2}. First we establish some lemmas.


\begin{lemma}\label{lem1} Permuting rows (or columns) of a matrix does not change membership in the class of lonesum matrices. \end{lemma}
\begin{proof} Suppose $M$ is not a lonesum matrix, and thus contains at least one of the forbidden minors as a submatrix. Then by permuting the rows (or columns) of $M$, either the same forbidden minor will still appear as a submatrix, or the other forbidden minor will now appear, thus the resulting matrix is also not a lonesum matrix. Similarly, if $M$ is a lonesum matrix and avoids the forbidden minors, then any permuting of rows (or columns) must also avoid the forbidden minors. 

\end{proof}

\begin{lemma}\label{lem2} If $M$ is a lonesum matrix, and the columns of $M$ are permuted such that they are arranged in decreasing order by column sum, then every row has the form $11...100...0$. \end{lemma} 
\begin{proof} Suppose that the columns of $M$ are arranged in decreasing order by column sum and there is a row that is not in this form, say row $j$. Then in row $j$ there is a 1 that appears to the right of a 0. Suppose this 1 is in column $k$. Since the columns are arranged in decreasing order by sum, there must be another row, say $\ell$, that contains a 1 in column $k-1$. Then in order to avoid the forbidden minors, the $(\ell, k)$ entry must be a 1. But then column $k$ has a higher sum than column $k-1$, thus a contradiction of the columns being in decreasing order.

 \end{proof}

\begin{lemma}\label{lem3} If $M$ is a lonesum matrix, and both the rows and the columns are permuted such that they are arranged in decreasing order by row sum and column sum, the resulting matrix has a staircase shape (i.e., the collection of 1's in the matrix forms a Young diagram). \end{lemma}
\begin{proof} Apply the previous lemma to both the rows and columns of M.
\end{proof}


\begin{proof}[Proof of Theorem \ref{thm:polyB}]
Let $M$ be an arbitrary lonesum matrix of size $n \times k$. Add a new row and a new column consisting of all 0s to obtain lonesum matrix $M'$ of size $(n+1) \times (k+1)$. This now guarantees that there is at least one row and at least one column consisting of all 0s. Then partition the rows and the columns of $M'$ according to the sum of their entries, so that two rows (or columns) are in the same class if they have the same row (or column) sum. The number of row classes and column classes must be the same by Lemma \ref{lem3}, and we call this value $m+1$, where $m$ denotes the number of nonzero classes. The row classes can be partitioned in $\stirling{n+1}{m+1}$ ways, and the column classes can be partitioned in $\stirling{k+1}{m+1}$ ways. Then order the nonzero row classes in $m!$ possible ways, and the nonzero column classes in $m!$ ways. Then using the two partitions and two orders, we can decode $M$. Summing over all of the possible values of $m$, we arrive at the result stated in \eqref{eq:Bnk}.
\end{proof}

\subsection{Bivariate generating function for the poly-Bernoulli numbers}\label{sec:GF}

As mentioned previously, the trivariate generating function in Theorem \ref{thm:main} is closely related to the bivariate (exponential in each variable) generating function for the poly-Bernoulli numbers $B_{n,k}$, and Theorem \ref{thm:main} can be viewed as a refinement of the following well-known result.

\begin{thm}
The generating function
\be\label{eq:Bxy}
B(x,y) = \sum_{k=0}^{\infty} \sum_{n=0}^{\infty} B_{n,k} \frac{x^n}{n!} \frac{y^k}{k!}
\ee
for the poly-Bernoulli numbers satisfies
\begin{equation}\label{BGF} B(x,y) = \frac{e^{x+y}}{e^x+e^y-e^{x+y}}.
 \end{equation}
\end{thm}

Let us review the proof of the analytic expression \eqref{BGF} for $B(x,y)$ as it will serve as a point of reference for the proof of Theorem \ref{thm:main}.

First, recall the generating function for the Stirling numbers of the second kind $\stirling{n}{k}$ for fixed $k$:
\begin{equation}\label{eq:stirlinggf2}
\sum_{n=k}^{\infty} \stirling{n}{k} \frac{x^n}{n!} = \frac{(e^x-1)^k}{k!}.
\end{equation}
Shifting the index,
$$\sum_{n=k}^{\infty} \stirling{n+1}{k+1} \frac{x^{n+1}}{(n+1)!} = \frac{(e^x-1)^{k+1}}{(k+1)!},$$ 
and differentiating, we obtain
\begin{equation}\label{eq:stirlinggf} 
\sum_{n=k}^{\infty} \stirling{n+1}{k+1} \frac{x^{n}}{n!} = \frac{e^x(e^x-1)^k}{k!}.
\end{equation}

In order to apply this, first substitute the formula \eqref{eq:formula} for $B_{n,k}$. Then change the order of summation.  
\begin{align*} B(x, y) &=\sum_{k=0}^{\infty} \sum_{n=0}^{\infty} B_{n,k} \frac{x^n}{n!} \frac{y^k}{k!} \\
 &= \sum_{k=0}^{\infty} \sum_{n=0}^{\infty} \left( \sum_{m=0}^{\infty} (m!)^2 \stirling{n+1}{m+1} \stirling{k+1}{m+1} \right) \frac{x^n}{n!} \frac{y^k}{k!} \\ 
&=  \sum_{m=0}^{\infty} (m!)^2 \sum_{n=0}^{\infty} \stirling{n+1}{m+1}\frac{x^n}{n!} \sum_{k=0}^{\infty} \stirling{k+1}{m+1}   \frac{y^k}{k!}  
\end{align*}

Substituting the generating function in Equation \eqref{eq:stirlinggf}, and simplifying, we obtain
\begin{align*} B(x,y) &= \sum_{m=0}^{\infty} (m!)^2 \sum_{n=0}^{\infty} \stirling{n+1}{m+1}\frac{x^n}{n!} \sum_{k=0}^{\infty} \stirling{k+1}{m+1}   \frac{y^k}{k!}  \\
&=  \sum_{m=0}^{\infty} (m!)^2 \left( \frac{e^x(e^x-1)^m}{m!} \right) \left( \frac{e^y(e^y-1)^m}{m!} \right) \\
&= \sum_{m=0}^{\infty} (e^{x+y})(e^{x+y}-e^x-e^y+1)^m.
\end{align*}
Recognizing this as a geometric series, we obtain the desired analytic expression: 
\begin{equation} B(x,y) = \frac{e^{x+y}}{e^x+e^y-e^{x+y}}.
 \end{equation}

We again emphasize what this means in terms of the coefficient extraction: 
$$[x^ny^k] \left( \frac{e^{x+y}}{e^x+e^y-e^{x+y}} \right) = B_{n,k},$$
i.e., the coefficient of $x^ny^k$ in the expansion of $\frac{e^{x+y}}{e^x+e^y-e^{x+y}}$ is exactly $B_{n,k}$. 
This generating function and its derivation will be revisited below in Section \ref{sec:proofGF}.

\section{Proof of Theorem \ref{thm:main} (enumeration of acyclic orientations refined by longest path length)}\label{sec:proofGF}

First, let us use the lonesum matrix associated to a given acyclic orientation in order to identify its longest path length.

Suppose we have an acyclic orientation of $K_{n,k}$, with parts $A$ and $B$ of sizes $n$ and $k$, respectively. We construct $M$, the  $n \times k$ matrix representation of this graph, as described in the proof of Theorem \ref{thm:bijection}, so that a 1 in the $(i,j)$ position represents an edge directed from vertiex $i$ in part $A$ to vertex $j$ in part $B$, and a 0 in the $(i,j)$ position represents an edge directed from vertex $j$ in part $B$ to vertex $i$ in part $A$. We will reorganize the rows and columns of $M$ to form a new matrix $M'$, with the rows and columns arranged in decreasing order according to their sums. Note that this rearranging of the matrix is equivalent to relabeling the vertices in the graph and thus preserves the length of the longest path.  This new matrix $M'$ will now have the staircase shape as described in the proof of Lemma \ref{lem3}, with $m$ being the number of nonzero row or column classes as in the proof of Theorem \ref{thm:polyB}. The length of the longest path, which we will denote $\ell$, will depend on this value of $m$ as well as the appearance of all-zero rows/columns.

Note that a path in the bipartite graph corresponds to a sequence of entries in our matrix that alternates between $0$s and $1$s with consecutive entries alternately sharing a common row or column (so that the corresponding consecutive edges are incident on a common vertex). We locate a longest path by starting in the top right corner of the reorganized matrix $M'$. 

First consider the case in which $M$, and so also $M'$, contains no row or column consisting entirely of 0s.  Select the top right entry of $M'$, which is necessarily a $1$ in this case, as the first entry in the sequence. This entry represents an edge in the associated (relabelled) graph directed from vertex $1$ in $A$ to vertex $k$ in $B$. Next we choose the first $0$ below this $1$ in the same column in $M'$, which represents an edge from $B$ to $A$ that starts at vertex $k$ in $B$. Continue moving down the ``staircase" shape and alternate between choosing the first $1$ to the left, from the same row as the previous entry, and then the first $0$ below, from the same column as the previous entry. We will stop when we have reached a 1 that does not have a $0$ below it, or we have reached the bottom row of the matrix. In this case we will have chosen $2m-1$ entries from the matrix, corresponding to a path of length $2m-1$.  We verify that this produces a path of maximal length, as well as address the cases with all-zero rows/columns in Lemma \ref{lemma:ell} below.

We will give an example to illustrate this process. Consider the following matrix as $M'$, that has already been reorganized with the rows and colmns in decreasing order by sum. We have no row or column consisting of only 0s in this matrix. The algorithm described above leads to the following sequence of entries (circled below). 
$$\left[ \begin{array}{ccccc}
     1 & 1 & 1 & 1 & \circled{1} \\
     1 & 1 & 1 & 1 & 0 \\
     1 & 1 & 1 & 1 & 0 \\
     1 & 1 & 0 & 0 & 0 \\
     1 & 0 & 0 & 0 & 0 \\
     1 & 0 & 0 & 0 & 0 \\
\end{array} \right]  \rightarrow 
\left[ \begin{array}{ccccc}
      1 & 1 & 1 & 1 & \circled{1} \\
     1 & 1 & 1 & 1 & \circled{0} \\
     1 & 1 & 1 & 1 & 0 \\
     1 & 1 & 0 & 0 & 0 \\
     1 & 0 & 0 & 0 & 0 \\
     1 & 0 & 0 & 0 & 0 \\
\end{array} \right]  \rightarrow \quad \cdots \quad \rightarrow
\left[ \begin{array}{ccccc}
      1 & 1 & 1 & 1 & \circled{1} \\
     1 & 1 & 1 & \circled{1} & \circled{0} \\
     1 & 1 & 1 & 1 & 0 \\
     1 & \circled{1} & 0 & \circled{0} & 0 \\
     \circled{1} & \circled{0} & 0 & 0 & 0 \\
     1 & 0 & 0 & 0 & 0 \\
\end{array} \right] $$

In this example, as described above, we have no row or column consisting of all 0s in the matrix $M'$. The value of $m$, which is the number of row (or column) classes, corresponds to the number of ``steps" in the staircase shape, in this case $m=4$. The length of the longest path for the graph corresponding to this matrix, which is equal to the number of entries we have chosen from the matrix, is $2m-1=7$. 

This brings us to state the following lemma. 

\begin{lemma}\label{lemma:ell}
Given an acyclic orientation on a complete bipartite graph, let $m$ denote the number of nonzero row/column classes in the associated lonesum matrix $M$.
Then the length $\ell$ of the longest path satisfies
$$\ell = \left\{ \begin{array}{cc}
     2m-1 & \text{if M contains no row or column of all 0s} \\
     2m & \text{if M contains a row or column of all 0s, but not both} \\
     2m+1 & \text{if M contains both a row and column of all 0s} 
\end{array} \right.$$
\end{lemma}

\begin{proof}
The case in which $M$ contains no row or column of all 0s was already described above. 

If we are in the case when $M$ contains a row or column of all $0$s, but not both, the above process produces a path of length $\ell = 2m$. Indeed, a column of all $0$s means the rightmost column of $M'$ is all $0$s.  After selecting the top right entry, a $0$, select the nearest $1$ to the left of it, and from there follow the rest of the procedure as described above.  This gives one additional entry at the beginning of the original sequence.  Similarly, a row of all $0$s leads to appending an additional entry at the end of the sequence. 

Similarly, the case where $M$ contains both a row and column of all $0$s increases the length by $2$ resulting in $\ell = 2m+1$. 

In the complete bipartite graph, we say that two vertices in part $A$ are in the same class if they have the same orientations on the edges with every vertex from part $B$, and similarly for the vertices in part $B$. In the matrix representation of the graph, two vertices of part $A$ in the same class will be represented by two rows in the same row class, and two vertices of part $B$ in the same class will be represented by two columns in the same column class. Note that using this algorithm it is guaranteed that exactly one vertex from each class in part $A$ and one vertex from each class in part $B$ are used. Note that the path produced by the above procedure has exactly one vertex from each class.

That the resulting path is indeed longest possible then follows from the fact that any path on an acyclic orientation of $K_{n,k}$ can contain at most one element from each class.

Let $a_1$ and $a_2$ be two vertices of $A$ that are in the same class. Then $a_1$ and $a_2$ have the same direction on the edges between each vertex in $B$. Suppose there exists a path $P$ that contains both $a_1$ and $a_2$. Note that a path in $G$ must alternate from vertices in $A$ to vertices in $B$. Then there is some $b_i \in B$ such that we have $(\dots, a_1, b_i, \dots )$ in the path $P$. Then since $a_1$ and $a_2$ are in the same class, the edge between $a_2$ and $b_i$ is also oriented toward $b_i$, so the path can not immediately go back to $a_2$. Then there is at least one vertex in each of $A$ and $B$ in between $a_1$ and $a_2$ in the path $P$. Suppose we have $(\dots, a_1, b_i, a_j, \dots , b_k, a_2, \dots)$ for some $a_j \in A$, $b_k \in B$. Since $a_1$ and $a_2$ are in the same class, the edge connecting $b_k$ to $a_1$ is also oriented toward $a_1$, thus creating a cycle. This is a contradiction. The same argument can be applied for vertices of $B$ that are in the same class. Thus, any path has at most one vertex from each class.
\end{proof}

It should be noted that the path itself is not necessarily unique.
For instance, returning to the above example, we see multiple alternative choices of entries that produces a path of equal length $\ell = 7$, such as (following the circled entries as before, from top right to bottom left)
$$
\left[ \begin{array}{ccccc}
      1 & 1 & 1 & 1 & \circled{1} \\
     1 & 1 & 1 & 1 & 0 \\
     1 & 1 & \circled{1} & 1  & \circled{0} \\
     1 & \circled{1} & \circled{0} & 0 & 0 \\
     \circled{1} & \circled{0} & 0 & 0 & 0 \\
     1 & 0 & 0 & 0 & 0 \\
\end{array} \right] .$$

\begin{proof}[Proof of Theorem \ref{thm:main}]
Let us define $b_{n,k}, c_{n,k},$ and $d_{n,k}$ as follows:
$$b_{n,k}(m) = (m!)^2 \stirling{n+1}{m+1} \stirling{k+1}{m+1}$$  
$$c_{n,k}(m) = (m!)^2 \stirling{n+1}{m+1} \stirling{k}{m}$$  $$d_{n,k}(m)=(m!)^2 \stirling{n}{m} \stirling{k}{m}.$$
Recall from the proof of Theorem \ref{thm:polyB} that $b_{n,k}$ counts the number of $n \times k$ lonesum matrices with $m$ nonzero row/column classes in which a row and/or column of 0s is allowed. Similarly, $c_{n,k}$ counts the number of $n \times k$ lonesum matrices with $m$ nonzero row/column classes in which a row of 0s is allowed, but not a column, and $d_{n,k}$ counts the number of $n \times k$ lonesum matrices with $m$ nonzero row/column classes in which no row or column of 0s is allowed. Note that by symmetry $c_{k,n}$ counts the number of $n \times k$ lonesum matrices with $m$ nonzero row/column classes in which a column of $0$s is allowed, but not a row.

Recall the exponential generating function for the Stirling numbers of the second kind, originally stated in Equations \eqref{eq:stirlinggf2} and \eqref{eq:stirlinggf}, which will be used in our results: 
$$ \sum_{n=k}^{\infty} \stirling{n}{k} \frac{x^n}{n!}= \frac{(e^x-1)^k}{k!},$$
and the shifted version
$$ \sum_{n=k}^{\infty} \stirling{n+1}{k+1} \frac{x^n}{n!} = \frac{e^x(e^x-1)^k}{k!}.$$

We will separate the counting into cases depending on the parity of the length of the longest path. First suppose that the length of the longest path is odd. By Lemma \ref{lemma:ell}, this can happen in two different ways; either $M$ contains no row or column of all $0$s, or $M$ contains both a row and a column of all $0$s.

Let $G_{n,k}(\ell)$ denote the number of acyclic orientations on a complete bipartite graph with parts of sizes $n$ and $k$ having longest path with length equal to $\ell$. Then first considering the odd case, when $\ell = 2m+1$, we have by Lemma \ref{lemma:ell} the following generating function: 
$$\sum_{n,k,m} G_{n,k}(2m+1) u^{2m+1} \frac{x^n}{n!} \frac{y^k}{k!} = \sum_{n,k,m} \left[ d_{n,k}(m+1) + (b_{n,k} - c_{n,k} - c_{k,n} + d_{n,k})(m) \right] u^{2m+1} \frac{x^n}{n!} \frac{y^k}{k!},$$
where we have used the inclusion-exclusion principle to obtain the count $(b_{n,k} - c_{n,k} - c_{k,n} + d_{n,k})(m)$ for the lonesum matrices with both a row and a column of all $0$s.

We want to obtain an analytic expression for this generating function. Due to space constraints, we will address each term separately. Using the known generating functions stated above, we have the following: 

Let $$D_{m+1}(x,y,u)=\sum_n \sum_k \sum_m d_{n,k}(m+1) u^{2m+1}\frac{x^n}{n!} \frac{y^k}{k!}.$$
Then by substituting in the expression for $d_{n,k}(m+1)$ and grouping by variables, we have \begin{align*}
D_{m+1}(x,y,u)  &=
    \sum_n \sum_k \sum_m (m+1)!^2 \stirling{n}{m+1} \stirling{k}{m+1} \frac{x^n}{n!} \frac{y^k}{k!} u^{2m+1} \\
    &= \sum_m (m+1)!^2 u^{2m+1} \sum_n \stirling{n}{m+1} \frac{x^n}{n!} \sum_k \stirling{k}{m+1} \frac{y^k}{k!}. 
\end{align*}

We can then use the known generating functions for the Stirling numbers, and after some simplification, we have 
\begin{align*}
 D_{m+1}(x,y,u)  &= \sum_m (m+1)!^2 u^{2m+1} \frac{(e^x-1)^{m+1}}{(m+1)!} \frac{(e^y-1)^{m+1}}{(m+1)!} \\
    &= \sum_m u^{2m+1} (e^x-1)^{m+1} (e^y-1)^{m+1} \\
    &= \sum_m u(e^x-1)(e^y-1) \left[ u^2 (e^x-1)(e^y-1) \right] ^m .
\end{align*}

Finally, by recognizing this as a geometric series, we obtain
\begin{equation}
    D_{m+1}(x,y,u) = \frac{u(e^x-1)(e^y-1)}{1-u^2(e^x-1)(e^y-1)}.
\end{equation}

Similarly, we can obtain the following analytic expressions for the remaining terms, shown here with some of the intermediate steps omitted (since the details are similar to the steps above):
\begin{align*}
    \sum_n \sum_k \sum_m b_{n,k}(m)u^{2m+1} \frac{x^n}{n!} \frac{y^k}{k!}  &= \sum_n \sum_k \sum_m (m!)^2 \stirling{n+1}{m+1} \stirling{k+1}{m+1} \frac{x^n}{n!} \frac{y^k}{k!} u^{2m+1} \\
    &= \frac{ue^xe^y}{1-u^2(e^x-1)(e^y-1)}
\end{align*}

\begin{align*}
    \sum_n \sum_k \sum_m c_{n,k}(m)u^{2m+1} \frac{x^n}{n!} \frac{y^k}{k!}  &= \sum_n \sum_k \sum_m (m!)^2 \stirling{n+1}{m+1} \stirling{k}{m} \frac{x^n}{n!} \frac{y^k}{k!} u^{2m+1} \\
    &= \frac{ue^x}{1-u^2(e^x-1)(e^y-1)}
\end{align*}

\begin{align*}
    \sum_n \sum_k \sum_m c_{k,n}(m) u^{2m+1}\frac{x^n}{n!} \frac{y^k}{k!}  &= \sum_n \sum_k \sum_m (m!)^2 \stirling{n}{m} \stirling{k+1}{m+1} \frac{x^n}{n!} \frac{y^k}{k!} u^{2m+1} \\
    &= \frac{ue^y}{1-u^2(e^x-1)(e^y-1)}
\end{align*}

\begin{align*}
     \sum_n \sum_k \sum_m d_{n,k}(m) u^{2m+1}\frac{x^n}{n!} \frac{y^k}{k!}  &= \sum_n \sum_k \sum_m (m!)^2 \stirling{n}{m} \stirling{k}{m} \frac{x^n}{n!} \frac{y^k}{k!} u^{2m+1} \\
    &= \frac{u}{1-u^2(e^x-1)(e^y-1)}
\end{align*}

Then by substituting these analytic forms back into the original equation, we have 
\begin{equation}\label{eq:odd}
   \sum_{n,k,m} G_{n,k}(2m+1) u^{2m+1} \frac{x^n}{n!} \frac{y^k}{k!} = \frac{2u(e^x-1)(e^y-1)}{1-u^2(e^x-1)(e^y-1)} .
\end{equation}


This gives us an analytic form of the generating function for the number of acyclic orientations of $K_{n,k}$ with longest path of length $2m+1$.

Next we will look at the case when the longest path length $\ell = 2m$ is even. By Lemma \ref{lemma:ell}, this happens when $M$ contains a row or column of all 0s, but not both. This leads to the following generating function for the even case (here the added $1$ accounts for the empty case $n=m=\ell=0$).
\be\label{eq:evenG}
\sum_{n,k,m} G_{n,k}(2m) u^{2m} \frac{x^n}{n!} \frac{y^k}{k!} = 1+ \sum_{n,k,m} \left[ c_{n,k}(m) + c_{k,n}(m) - 2 d_{n,k}(m) \right] u^{2m} \frac{x^n}{n!} \frac{y^k}{k!},
\ee
where we have again used a form of the inclusion-exclusion principle.  Namely, to count the matrices with a row or a column consisting of all 0s, but not both, we take those which allow a row of all $0$s ($c_{n,k}$), add those which allow a column of all $0$s ($c_{k,n}$), and then subtract those which do not allow a row or column of all 0s ($d_{n,k}$) twice (since they have been counted twice).

Next we obtain an analytic form of the generating function using steps similar to the odd case.

\begin{align*}
    \sum_n \sum_k \sum_m c_{n,k}(m)u^{2m} \frac{x^n}{n!} \frac{y^k}{k!}  &= \sum_n \sum_k \sum_m (m!)^2 \stirling{n+1}{m+1} \stirling{k}{m} \frac{x^n}{n!} \frac{y^k}{k!} u^{2m} \\
    &= \frac{e^x}{1-u^2(e^x-1)(e^y-1)}
\end{align*}

\begin{align*}
    \sum_n \sum_k \sum_m c_{k,n}(m) u^{2m}\frac{x^n}{n!} \frac{y^k}{k!}  &= \sum_n \sum_k \sum_m (m!)^2 \stirling{n}{m} \stirling{k+1}{m+1} \frac{x^n}{n!} \frac{y^k}{k!} u^{2m} \\
    &= \frac{e^y}{1-u^2(e^x-1)(e^y-1)}
\end{align*}

\begin{align*}
     \sum_n \sum_k \sum_m d_{n,k}(m) u^{2m}\frac{x^n}{n!} \frac{y^k}{k!}  &= \sum_n \sum_k \sum_m (m!)^2 \stirling{n}{m} \stirling{k}{m} \frac{x^n}{n!} \frac{y^k}{k!} u^{2m} \\
    &= \frac{1}{1-u^2(e^x-1)(e^y-1)}
\end{align*}

Using these in \eqref{eq:evenG}, we obtain 
\begin{equation}\label{eq:even}
    \sum_{n,k,m} G_{n,k}(2m) u^{2m} \frac{x^n}{n!} \frac{y^k}{k!} = 1+ \frac{e^x+e^y-2}{1-u^2(e^x-1)(e^y-1)} .
\end{equation}

Then adding the odd \eqref{eq:odd} and even \eqref{eq:even} parts and simplifying, we obtain the desired analytic expression for the full generating function $$F(x,y,u)= \frac{e^{x+y} - (u-1)^2(e^{x}-1)(e^{y}-1)}{1-u^2(e^x-1)(e^y-1)}.$$
\end{proof}

\section{Proof of Theorem \ref{thm:asymp} (Gaussian limit for longest path length)}\label{sec:AC}
The purpose of this section is to confirm that the distribution of the length of the longest path $\ell$ in the case when $n=k$ is asymptotically Gaussian. We also determine precise asymptotics for the mean and variance of the distribution. 

First, we recall the definition of a \textit{Gaussian distribution}, or normal distribution, which has the probability density function 
$$f(x)=\frac{1}{\sigma \sqrt{2\pi}}e^{-\frac{1}{2} \left( \frac{x-\mu}{\sigma} \right)^2} ,$$
where $\mu$ is the mean of the distribution (which is also its median and mode), $\sigma$ is the standard deviation, and $\sigma^2$ is the variance. The \textit{standard Gaussian} is the special case when $\mu=0$ and $\sigma = 1$.

We also recall what it means for a sequence of random variables to converge in distribution. The random variables $X_n$ are said to \textit{converge in distribution} to a variable $Y$, if, for each real number $t$, one has 
$$ \lim_{n\rightarrow \infty} \mathbb{P}(X_n \leq t) = \mathbb{P} (Y\leq t).$$ 

The following result \cite{Flajolet} reduces the task of proving Theorem \ref{thm:asymp} to that of showing that the sequence of probability generating functions $p_n(u)$ admits the so-called quasi-power condition \eqref{eq:quasipow} which asserts that a certain asymptotic holds uniformly near $u=1$.

In the statement of the theorem, for a function $f(u)$ which is analytic at $u=1$ with $f(1) \neq 0$, we define 
\be\label{eq:m}
m(f)=\frac{f'(1)}{f(1)}
\ee
and
\be\label{eq:v}
v(f)= \frac{f''(1)}{f(1)} + \frac{f'(1)}{f(1)} - \left( \frac{f'(1)}{f(1)} \right) ^2.
\ee

\begin{thm}[{Quasi-powers Theorem~\cite[Thm IX.8]{Flajolet}}]\label{thm:Flaj}
Let the $X_n$ be nonnegative discrete random variables with probability generating functions $p_n(u)$. Assume that, uniformly, in a fixed complex neighborhood of $u=1$, for sequences $\beta_n, \kappa_n \rightarrow + \infty$, there holds 
\be\label{eq:quasipow}
p_n(u)=A(u) \cdot B(u)^{\beta_n} \left( 1+ O \left( \frac{1}{\kappa_n} \right) \right)
\ee
where $A(u), B(u)$ are analytic at $u=1$, and $A(1)=B(1)=1$. Assume finally that $B(u)$ satisfies the so-called ``variability condition" 
$$v \left( B(u) \right) \coloneqq B''(1)+B'(1)-B'(1)^2 \neq 0.$$
Under these conditions, the mean and variance of $X_n$ satisfy
\be\label{eq:mean} \mu_n = \mathbb{E}(X_n) = \beta_n m(B(u)) + m(A(u)) + O(\kappa_n^{-1})
\ee
\be\label{eq:var} \sigma_n^2 = \mathbb{V}(X_n) = \beta_n v(B(u)) + v(A(u)) + O(\kappa_n^{-1}).
\ee

The distribution of $X_n$ is, after standardization, asymptotically Gaussian, and the speed of convergence to the Gaussian limit is $O(\kappa_n^{-1}+\beta_n^{-1})$: 
$$ \mathbb{P} \left\{ \frac{X_n-\mathbb{E}(X_n)}{\sqrt{(\mathbb{V}(X_n)}} \leq x \right\} = \Phi (x) + O\left( \frac{1}{\kappa_n} + \frac{1}{\sqrt{\beta_n}} \right),$$ 
where $\Phi(x)$ is the commulative distribution of a standard normal, i.e.,
$$ \Phi(x) = \frac{1}{\sqrt{2\pi}} \int_{- \infty}^{x} e^{-\omega^2/2} d\omega.$$
\end{thm}

The quasi-power condition \eqref{eq:quasipow} needed for applying the above theorem will be established using Analytic Combinatorics in Several Variables (ACSV), namely, the following result of Pemantle and Wilson \cite{PemantleWilson} that provides asymptotics for coefficients of multivariate generating functions of the form $F = G/H$, where $G$ and $H$ are entire functions.  The singularity set $$\cV = \{(x,y)\in\C^2 : H(x,y)=0 \}$$ is important in determining asymptotics. A singularity $(p,q) \in \cV$ is called \textit{minimal} if there does not exist $(x,y) \in \C^2$ such that $H(x,y)=0$ with $|x| \leq |p|$ and $|y| \leq |q|$, with one of the inequalities being strict. Equivalently, the minimal singularities are the elements of $\cV$ on the closure of the power series domain of convergence $\cD$ of $F(x,y)$. A singularity $(p,q) \in \cV$ is \textit{strictly minimal} if it is minimal and there are no other singularities with the same coordinate-wise modulus.



\begin{thm}[{Pemantle and Wilson~\cite[Thm. 9.5.7]{PemantleWilson}, Pemantle, Wilson, and Melczer \cite[Thm. 9.4]{MPW}}]  
\label{thm:PW}
Let $F(x,y)$ be the ratio of entire functions $G,H$ with singularity set $\cV = \left\{(x, y) \in \C^2: H(x, y) = 0\right\}$ and power series expansion $$F(x, y) = \sum_{r, s \geq 0} A_{r,s} x^ry^s$$ convergent near the origin. Suppose there is a compact set $\cR$ of directions (i.e., a set of unit direction vectors) such that
\begin{enumerate}  
    \item[$(i)$] for each $r,s>0$ having direction in $\cR$ there exists a unique minimal point $(x_{r,s},y_{r,s}) \in \cV$ solving the system
\begin{equation} H(x,y) = sxH_x(x,y) - ryH_y(x,y) = 0, \label{eq:crit} \end{equation}
    and this point is strictly minimal;
    \item[$(ii)$]  The point $(x_{r,s},y_{r,s})$ varies smoothly with $r,s$ restricting its direction to $\cR$;
    \item[$(iii)$]  The point $(x_{r,s},y_{r,s})$ is not in the zero set of $G(x,y)$ nor in the zero set of
        \be\label{eq:Q}
        Q(x,y):= -y^2H_y^2xH_x-yH_yx^2H_x^2-x^2y^2(H_y^2H_{xx}+H_x^2H_{yy}-2H_xH_yH_{xy}). 
        \ee
\end{enumerate}  
Then as $r, s \to \infty$
\begin{equation}\label{eq:PW}
 A_{r,s} = \left( G(x,y) + O(s^{-1}) \right) \frac{1}{\sqrt{2\pi}}x^{-r}y^{-s}\sqrt{\frac{-yH_y(x,y)}{sQ(x,y)}},
\end{equation}
where the constant in the error term $O(s^{-1})$ is uniform over $r,s$ with direction in $\cR$.
\end{thm}

Specifically, we will apply this result along the diagonal direction $r=s$.
In order to apply Theorem \ref{thm:PW}, we will need the following lemma ensuring strict minimality.

\begin{lemma} \label{lem:SM}
There exists $\delta >0$ such that for all $ u \in \mathbb{C}$ satisfying  $\vert u-1 \vert < \delta$, the point $(a,b) = \left( \log \left( 1 + \frac{1}{u} \right) , \log \left( 1 + \frac{1}{u} \right) \right)$ is strictly minimal with respect to $F(x,y,u)$, viewed as a bivariate generating function with $u$ as a parameter. 
\end{lemma} 

We defer the (somewhat technical) proof of the lemma in favor of showing how it is used to prove the main theorem.

\begin{proof}[Proof of Theorem \ref{thm:asymp}]
Recall the expression in terms of coefficient extraction for the probability generating function $p_n(u)$ for the case when both parts of the bipartite graph are of size $n$: 
$$p_{n}(u)= \frac{[x^ny^n]F(x,y,u)}{[x^ny^n]F(x,y,1)}, \quad F(x,y,u)= \frac{e^{x+y} - (u-1)^2(e^{x}-1)(e^{y}-1)}{1-u^2(e^x-1)(e^y-1)} = \frac{G(x,y,u)}{H(x,y,u)},$$
where $G(x,y,u)=e^{x+y} - (u-1)^2(e^{x}-1)(e^{y}-1)$, $H(x,y,u)= 1-u^2(e^x-1)(e^y-1)$ are each entire functions.
In order to use Theorem \ref{thm:Flaj} we need to prove a quasi-power statement for our probability generating function.  Specifically, we will show
\be\label{eq:quasipow2}
p_n(u) = A(u)\cdot B(u)^n \left( 1+ O\left( \frac{1}{n} \right) \right) .
\ee
To do this we will utilize Theorem \ref{thm:PW}. Since we are considering the diagonal case $n=k$, condition \eqref{eq:crit} simplifies to $x=y$, and we find a unique candidate in the intersection $\{ x=y \} \cap \left\{ (e^x-1)(e^y-1)= \frac{1}{u^2} \right\}$, which is $x=y=\log \left( 1+ \frac{1}{u} \right)$.  This point is strictly minimal by Lemma \ref{lem:SM}, and $G(\log \left( 1+ \frac{1}{u} \right),\log \left( 1+ \frac{1}{u} \right),u ) = 4/u $. We compute the partial derivatives of $H$ and evaluate them at $x=y=a(u):=\log \left( 1+ \frac{1}{u} \right)$.
$$H_x(a(u),a(u),u) = H_y( a(u),a(u),u) = H_{xx}(a(u),a(u),u) = H_{yy}(a(u),a(u),u) = -(u+1)$$
$$H_{xy}(a(u),a(u),u) = -(u+1)^2$$

$Q$ defined in \eqref{eq:Q}
$$Q(a(u),a(u),u) = 2(u+1)^3\log^3 \left( 1+ \frac{1}{u} \right) \left[ 1 - u\log \left( 1+ \frac{1}{u} \right) \right].$$
This verifies conditions (i) and (iii) in Theorem \ref{thm:PW} (and condition (ii) is trivially satisfied since we are considering just the diagonal direction, so the set $\cR$ is a singleton).

Applying Theorem \ref{thm:PW} to both numerator and denominator of $p_n(u)$, we obtain
$$ [x^ny^n]F(x,y,u)= (n!)^2 \frac{2}{u(u+1)}\sqrt{\frac{1}{n\pi(1-u\log \left( 1+\frac{1}{u} \right) )}}\left( \frac{1}{\log \left( 1+\frac{1}{u} \right) } \right) ^{2n+1} \left( 1 + O(n^{-1})  \right), \quad \text{as \, } n \rightarrow \infty. 
$$
$$[x^ny^n]F(x,y,1)= (n!)^2 \sqrt{\frac{1}{n\pi(1-\log 2)}}\left( \frac{1}{\log 2} \right) ^{2n+1} \left( 1 + O(n^{-1})  \right), \quad \text{as \, } n \rightarrow \infty. 
$$
Hence, we have
$$p_n(u)=\left( \frac{\log 2}{\log \left( 1+\frac{1}{u} \right) } \right) ^{2n} \left( \frac{2 \log 2}{\log \left( 1+\frac{1}{u} \right) u (u+1)}  \sqrt{\frac{1-\log 2}{1-u \log \left( 1+\frac{1}{u} \right) }} + O\left( \frac{1}{n} \right) \right),$$
which takes the desired form of the quasi-power statement \eqref{eq:quasipow2}, namely with 
$$A(u) = \frac{2 \log 2}{\log \left( 1+\frac{1}{u} \right) u (u+1)} \cdot \sqrt{\frac{1-\log 2}{1-u \log \left( 1+\frac{1}{u} \right) }} , $$ 
and $$B(u) = \left( \frac{\log 2}{\log \left( 1+\frac{1}{u} \right) } \right) ^2 .$$
We conclude that the longest path length converges in distribution (after recentering and rescaling) to a Gaussian as stated in the theorem.  
Moreover, we obtain asymptotics for the mean and variance by applying the statements \eqref{eq:mean} and \eqref{eq:var}.
First, recalling the definitions \eqref{eq:m} and \eqref{eq:v} we compute
\be\label{eq:mA}
m(A(u))=\frac{A'(1)}{A(1)} = 
\frac{8(\log 2)^2-9 \log 2 + 2}{4 \log 2 (1-\log 2)},\ee
\be\label{eq:mB}
m(B(u))=\frac{B'(1)}{B(1)} = \frac{1}{\log 2}
,\ee
\be\label{eq:vA}
v(A(u))= \frac{A''(1)}{A(1)} + \frac{A'(1)}{A(1)} - \left( \frac{A'(1)}{A(1)} \right) ^2 = \frac{-2 (\log 2)^4 + (\log 2)^3 + 2 (\log 2)^2 - 6 \log 2 + 2}{8 (\log 2)^2 (1- \log 2)^2},
\ee
and
\be\label{eq:vB}
v(B(u))= \frac{B''(1)}{B(1)} + \frac{B'(1)}{B(1)} - \left( \frac{B'(1)}{B(1)} \right) ^2 = \frac{ (1-\log 2)}{2 (\log 2)^2}.
\ee
Then using these computations in \eqref{eq:mean}, \eqref{eq:var}, we obtain the desired asymptotics:
$$\mu_n =  \frac{n}{\log 2} + \frac{8(\log 2)^2-9 \log 2 + 2}{4 \log 2 (1-\log 2)} + O(n^{-1}), \quad \text{as \, } n \rightarrow \infty$$ and 
$$\sigma_n^2 = \frac{n (1-\log 2)}{2 (\log 2)^2} + \frac{-2 (\log 2)^4 + (\log 2)^3 + 2 (\log 2)^2 - 6 \log 2 + 2}{8 (\log 2)^2 (1 - \log 2)^2} + O(n^{-1}), \quad \text{as \, } n \rightarrow \infty.$$ 
This concludes the proof.
\end{proof}
It remains to prove Lemma \ref{lem:SM}.  First, we establish some preliminary curvature estimates for the boundary of a certain planar region that will need to be considered.

\begin{lemma} \label{lem:curve}
Fix $u \in \C$ near $u=1$, and define $r = \left| \log \left( 1+\frac{1}{u} \right) \right|$.  Let $\gamma=\left\{ u(e^{re^{i\theta}}-1), \  \theta \in [- \pi, \pi ] \right\}$.  Then for $u$ sufficiently close to $u=1$, we have
\begin{enumerate}
\item[$(i)$] the curvature $K_\gamma(\theta)$ of $\gamma$ is positive for all $\theta \in [-\pi,\pi]$ (i.e., $\gamma$ is convex) and is given by $$K_{\gamma}(\theta) = \frac{1+r \cos \theta}{\vert u \vert re^{r \cos \theta} }.$$

\item[$(ii)$] the function $K_{\gamma}(\theta)$ has a local minimum at $\theta = 0$, and we have 
\begin{equation} \label{eq:min0}
\min_{\theta \in [- \frac{\pi}{2}, \frac{\pi}{2}]} K_{\gamma}(\theta) = K_{\gamma}(0) = \frac{1+ r}{|u|r e^r} > 1. 
\end{equation}

\item[$(iii)$] the function $K_{\gamma}(\theta)$ has a global minimum at $\theta = \pi$, and we have 
\begin{equation} 
\min_{\theta \in [-\pi, \pi]} K_{\gamma}(\theta) = K_{\gamma}(\pi) = \frac{1-r}{|u| r e^{-r}} < 1.
\end{equation}\label{eq:minpi}
\end{enumerate}
\end{lemma}

\begin{proof}
Let $f(\theta)=u \left( e^{re^{i\theta}} -1 \right)$, so that $\gamma$ is the image of $f$. Then by vector calculus (and using the convenient complex expression for the cross product $\alpha \times \beta = \text{Im} \{ \alpha \beta \}$ of complex numbers $\alpha$, $\beta$ viewed as vectors), $$K_\gamma(\theta) = \frac{f'(\theta)\times f''(\theta)}{\vert f'(\theta) \vert ^3} = \frac{-\text{Im} \{ f'(\theta)\overline{f''(\theta)} \}}{\vert f'(\theta) \vert ^3} = \frac{1+r \cos \theta}{\vert u \vert re^{r \cos \theta} }.$$

Then notice that we can write the function $K_{\gamma}(\theta) = g(r \cos \theta)$, where $$g(t)= \frac{1+t}{\vert u \vert re^t}.$$ 
Then $$g'(t) = \frac{ure^t - (1+t) ure^t}{(ure^t)^2} = \frac{-t}{ure^t}$$ 
which is positive for $t<0$ and negative for $t>0$, thus $g(t)$ is monotone increasing over the interval $(-1,0)$ and monotone decreasing over the interval $(0,1)$. Then $K_{\gamma}(\theta) = g(r \cos \theta)$ is increasing for $\theta \in ( - \pi, - \frac{\pi}{2} ) \cup (0, \frac{\pi}{2} )$ and is decreasing for $\theta \in ( - \frac{\pi}{2}, 0 ) \cup ( \frac{\pi}{2}, \pi )$. So $K_{\gamma}(\theta)$ has local minima at $\theta = - \pi, 0, \text{ and } \pi$. Comparing the values of $K_{\gamma}(0)= \frac{1+ r}{|u|r e^r} \approx \frac{1+ \log 2}{2 \log 2} = 1.22134...$ and $K_{\gamma}(\pi) = K (- \pi) = \frac{1-r}{|u| r e^{-r}}  \approx 2 \frac{1-\log 2}{ \log 2 } = 0.88539...$ we see that $K_{\gamma}(0)$ is a local minimum, and the absolute minimum over the interval $[ - \pi, \pi ]$ occurs at $K_{\gamma}(\pi).$

\end{proof}

\begin{proof}[Proof of Lemma \ref{lem:SM}]
We want to show that the point $x=y=\log \left( 1+ \frac{1}{u} \right)$ is the only point in the polydisk 
$$D\left( \left| \log \left( 1+ \frac{1}{u} \right) \right|, \left| \log \left( 1+ \frac{1}{u} \right) \right| \right) $$ 
which satisfies $1-u^2(e^x-1)(e^y-1)=0$. In other words, we need to show that the set $$A= \left\{ u(e^x-1), \ \vert x \vert \leq \left| \log \left( 1+ \frac{1}{u} \right) \right| \right\}$$ intersects the set $$B = \left\{ \frac{1}{u(e^y-1)}, \ \vert y \vert \leq \left| \log \left( 1+ \frac{1}{u} \right) \right| \right\}$$ only at the point $w=1$. 


Consider points $w \in A \cap B$ such that $\vert w \vert \neq 1$. Note that if $w \in A$, then $\frac{1}{w} \in B$, and similarly if $w \in B$, then $\frac{1}{w} \in A$, and so $w \in A \cap B$ if and only if $\frac{1}{w} \in A \cap B$. Thus, it is enough to consider only the case when $\vert w \vert \geq 1$, since if $\vert w \vert \leq 1$, then $\left| \frac{1}{w} \right| \geq 1$, and hence it suffices to establish the following.

{\bf Goal.}  It suffices to show that $A \cap \left\{ \vert w \vert \geq 1, w \neq 1 \right\}$ is mapped by inversion $w \mapsto 1/w$ to a region that has empty intersection with the set $A$. 

With $u$ fixed in a small neighborhood $|u-1| < \delta$ of $u=1$ (with the smallness of $\delta>0$ to be specified later), we use polar representation $w = r e^{i \theta}$ with $r = \left| \log \left(1 + \frac{1}{u}\right) \right|$ to parameterize the boundary of $A$ which is the curve $\gamma=\left\{ u(e^{re^{i\theta}}-1), \  \theta \in [- \pi, \pi ] \right\}$ appearing in Lemma \ref{lem:curve}.  Note that $r = \left| \log \left(1+ \frac{1}{u} \right) \right|$ approaches $\log 2$ as $\delta \rightarrow 0$ since $u$ approaches $1$.

Define the circle $C_1$ to be tangent to $\gamma$ at $w=1$ with curvature $K_{\gamma}(0)$. 

{\bf Claim 1.} The set $A \cap \left\{ \vert w \vert \geq 1, w \neq 1 \right\}$ is contained entirely in the interior of $C_1$.

To prove the claim, it is enough to show that $\gamma \cap \left\{ |w|\geq 1, w \neq 1 \right\}$ is contained in the interior of $C_1$ (indeed, failure of the claim results in $A$ having boundary points in the exterior of $C_1$ satisfying $|w| \geq 1$).  Next notice that by Lemma \ref{lem:curve} ($ii$) $C_1$ contains the part of the curve $\gamma$ for $- \frac{\pi}{2} \leq \theta \leq \frac{\pi}{2}$. Hence, it suffices to show that the part of $\gamma$ outside the unit circle is confined to the points mapped from the interval $- \frac{\pi}{2} \leq \theta \leq \frac{\pi}{2}$.
To see this, we will give an upper bound for $\vert u (e^{re^{i\theta}}-1) \vert$ on the interval $\frac{\pi}{2} \leq \theta \leq \pi$. Notice that
\begin{align*}
u (e^{re^{i\theta}}-1) &= u\left( e^{r \cos \theta} e^{ri \sin \theta} -1 \right) \\
&= u \left( e^{r \cos \theta} \left( \cos(r \sin \theta) + i \sin (r \sin \theta) \right) -1 \right) \\
&= u \left( e^{r \cos \theta} \cos(r \sin \theta) -1 + i e^{r \cos \theta} \sin (r \sin \theta) \right).
\end{align*}

Then for $\frac{\pi}{2} \leq \theta \leq \pi$, looking at this expression in pieces we have the following inequalities: 
$$-1 < \cos \theta < 0 \text{ which implies that } e^{-r} < e^{r \cos \theta} < 1$$
$$0 < r \sin \theta < r \text{ which implies that } \cos r < \cos(r\sin \theta) <1$$
therefore 
$$0 < e^{r \cos \theta} \cos(r \sin \theta) -1 < \vert e^{-r} \cos r -1 \vert.$$
Similarly, 
$$0 < r \sin \theta < r \text{ which implies that } 0 < \sin (r  \sin \theta) < \sin r$$
therefore 
$$ 0 < e^{r \cos \theta} \sin (r \sin \theta) < \sin r. $$

Then since $\vert a + ib \vert = \sqrt{a^2+b^2}$, and using the approximations $u \approx 1$ and $r \approx \log 2$, we have the following upper bound: 
\begin{align*}
\vert u (e^{re^{i\theta}}-1) \vert &\leq \sqrt{\left( e^{-r} \cos r -1 \right) ^2 + (\sin r)^2} \\
&\approx \sqrt{\left(\frac{1}{2} \cos ( \log 2) -1 \right)^2 + \sin^2(\log 2)} \\
&\approx 0.887\dots 
\end{align*}

Thus on the interval $\frac{\pi}{2} \leq \theta \leq \pi$, and also the interval $-\pi \leq \theta \leq - \frac{\pi}{2}$ by symmetry, $\vert u (e^{re^{i\theta}}-1) \vert <1$ and so $\gamma$ remains inside the unit circle, and it is only points of $\gamma$ in the image of the interval $- \frac{\pi}{2} \leq \theta \leq \frac{\pi}{2}$ that lie outside the unit circle. 

Define the circle $C_2$ to be the image of $C_1$ under the mapping $z \rightarrow \frac{1}{z}$. Since this mapping is a M\"{o}bius transformation, the image of $C_1$ must be a circle, and by conformality it is also tangent to $\gamma$ at $w=1$. 

{\bf Claim 2.} The circle $C_2$ contains the set $A \setminus \{1\}$ entirely in its interior. 

To prove the claim, we show that the curvature of $C_2$ is smaller than the minimal curvature of $\gamma$.

First consider the case when $u=1$ (the case of $u$ near $1$ will then follow by continuity), the curvature of $C_1$ is given by $K_{\gamma}(0)=\frac{1+ \log 2}{2 \log 2} \approx 1.2213\dots$ For a circle, the curvature is the reciprocal of the radius, so the radius of $C_1$ is $R_1 = \frac{2 \log 2}{1+ \log 2}$. We know $C_1$ intersects the real axis at $z=1$, and by symmetry (which holds in the case $u=1$) we obtain the other point it intersects the real axis, $\alpha$, by subtracting twice the radius to obtain, $\alpha = 1 - 2 R_1$. Under the inversion, the image of this point falls on the real axis at $\frac{1}{\alpha} = \frac{1}{1-2 R_1}$, and this determines the value $R_2 = \frac{1}{2}\left( 1 - \frac{1}{1-2 R_1} \right) = \frac{R_1}{2R_1 - 1}$ of the radius of $C_2$, and thus the curvature, $K_2 = 1/R_2 = 0.7786\dots$ We note that by Lemma \ref{lem:curve} ($iii$) this is smaller than the absolute minimum of the curvature of $\gamma$ and thus $\gamma$ is completely contained within the circle $C_2$. Since the curvature of $C_2$, as well as the minimum curvature of $\gamma$, changes continuously with respect to changes in the value of $u$, for $u$ sufficiently close to $1$ (i.e., for $\delta > 0$ sufficiently small), the same comparison of curvatures holds, so that the entirety of $\gamma$ is completely inside of the circle $C_2$, as stated in the claim.

Now we apply both claims to finish the proof of the lemma. Since the interior of $C_1$ is mapped by inversion to the exterior of $C_2$, Claim 1 implies that the set $A \cap \left\{ \vert w \vert \geq 1, w \neq 1 \right\}$ is mapped by inversion to a set contained in the exterior of the circle $C_2$.  By Claim 2, this then establishes that $A \cap \left\{ \vert w \vert \geq 1, w \neq 1 \right\}$ is mapped by inversion $w \mapsto 1/w$ to a region that has empty intersection with set $A$. This verifies the above-stated Goal and concludes the proof of the lemma.
\end{proof}

\section{Open Problems}\label{sec:concl}

Let us point out some natural directions for further research.  One obvious direction to consider is Cameron's question (mentioned at the beginning of the paper) restricted to other settings besides complete bipartite graphs, for instance, what can be said about the distribution of the longest path length for acyclic orientations on complete multi-partite graphs?  What about complete graphs?  It also seems natural to ask about random acyclic orientations on a random graph (sampling first the random graph from some model, such as the Erd\"os-Renyi model)?

Staying within the setting of complete bipartite graphs, another set of open problems to consider concerns other (non-diagonal case) asymptotic regimes.  For instance, is the outcome still Gaussian in other directions of the array, i.e., when $n\rightarrow \infty$ with $k= \alpha \cdot n$ for all $0<\alpha<1$ fixed? We suspect an affirmative answer, with the expectation that Theorem \ref{thm:PW} still applies and seems to carry a similar structure, albeit less explicit and more computationally involved.  Hence, it may still be possible to establish the quasi-power condition (even if the asymptotic mean and variance can not be specified explicitly).
Outside the setting where $k$ is linear in $n$, one may first consider the case when $k$ is fixed, and $n \rightarrow \infty$.  We expect a singular limit in this case with the distribution concentrating at the maximal longest length possible $2k-1$.  This raises the question of whether there are interesting (with non-singular and non-Gaussian limit) intermediate regimes,  when $n,k \rightarrow \infty$ with $k$ growing more slowly than linearly in $n$.

\bibliographystyle{abbrv}
\bibliography{ref}

\end{document}